\documentclass[
11pt, 
a4paper, 
twoside, 
headinclude,footinclude, 
]{amsart}

\usepackage[T1]{fontenc} 
\usepackage[utf8]{inputenc} 
\usepackage{graphicx} 
\graphicspath{{Figures/}} 
\usepackage{subfig} 
\usepackage{amsmath,amssymb,amsthm,stmaryrd,mathtools} 
\numberwithin{equation}{section}
\usepackage{varioref} 
\usepackage[shortlabels]{enumitem} 
\usepackage{tikz}
\usepackage{shuffle}
\usepackage{comment}
\usepackage{hyperref}

\usetikzlibrary{matrix}
\usetikzlibrary{decorations.markings}

\theoremstyle{plain} 
\newtheorem{theorem}{Theorem}[section]
\newtheorem*{thma}{Theorem A}
\newtheorem*{thmb}{Theorem B}
\newtheorem*{thmc}{Theorem C}

\newtheorem{proposition}[theorem]{Proposition}
\newtheorem{corollary}[theorem]{Corollary}
\newtheorem{lemma}[theorem]{Lemma}

\newtheorem*{question*}{Question}

\theoremstyle{definition} 
\newtheorem{definition}[theorem]{Definition}
\newtheorem{example}[theorem]{Example}
\newtheorem{question}[theorem]{Question}

\theoremstyle{remark} 
\newtheorem{remark}[theorem]{Remark}

\newcommand{\Z}{\mathbb Z}

\newcommand{\DD}{\mathsf D}

\newcommand{\Se}{\mathbb{S}}
\newcommand{\bo}{\operatorname{b}\nolimits}

\DeclareMathOperator{\sCM}{\underline{\mathsf{CM}}}

\newcommand{\Hom}{\operatorname{Hom}\nolimits}
\newcommand{\End}{\operatorname{End}\nolimits}
\newcommand{\Ext}{\operatorname{Ext}\nolimits}
\newcommand{\Tor}{\operatorname{Tor}\nolimits}

\newcommand{\otimesk}{\otimes_k}
\newcommand{\otimesS}{\otimes_S}

\newcommand{\bd}{\DD^{\bo}}
\newcommand{\Mod}{\mathsf{Mod}\,}

\newcommand{\RHom}{\mathbf{R}\strut\kern-.2em\operatorname{Hom}\nolimits}
\newcommand{\Lotimes}{\mathop{\stackrel{\mathbf{L}}{\otimes}}\nolimits}
\newcommand{\Image}{\operatorname{Im}\nolimits}

\DeclareMathOperator{\Tot}{Tot}

\author{Louis--Philippe Thibault}
\address{Institutt for matematiske fag, NTNU, 7491 Trondheim, Norway}
\email{louis.thibault@ntnu.no}
\dedicatory{Dedicated to the memory of Ragnar--Olaf Buchweitz}

\begin{document}

\title{Preprojective algebra structure on skew-group algebras}



\begin{abstract}
We give a class of finite subgroups $G<SL(n, k)$ for which the skew-group algebra $k[x_1,\ldots, x_n]\#G$ does not admit the grading structure of a higher preprojective algebra. Namely, we prove that if a finite group $G<SL(n, k)$ is conjugate to a finite subgroup of $SL(n_1, k)\times SL(n_2, k)$, for some $n_1, n_2\geq 1$, then the skew-group algebra $k[x_1,\ldots,x_n]\#G$ is not Morita equivalent to a higher preprojective algebra. This is related to the preprojective algebra structure on the tensor product of two Koszul bimodule Calabi--Yau algebras. We prove that such an algebra cannot be endowed with a grading structure as required for a higher preprojective algebra. 

Moreover, we construct explicitly the bound quiver of the higher preprojective algebra over a finite-dimensional Koszul algebra of finite global dimension. We show in addition that preprojective algebras over higher representation-infinite Koszul algebras are derivation-quotient algebras whose relations are given by a superpotential.  
\end{abstract}

\maketitle
\tableofcontents

\setcounter{tocdepth}{2}


\section{Introduction}

Preprojective algebras were first defined by Gel$'$fand and Ponomarev \cite{GP79} to study the representation theory of finite-dimensional hereditary algebras. Baer, Geigle and Lenzing \cite{BGL87} later proposed an equivalent definition of the preprojective algebra as a tensor algebra over $\Lambda$ of the inverse Auslander--Reiten translation: 
\[
	\Pi(\Lambda) := T_{\Lambda}\Ext_{\Lambda}^1(D\Lambda, \Lambda).
\]	
In the setting of Iyama's program of higher-dimensional Auslander--Reiten theory, this construction was generalized to algebras of higher global dimension. They are characterized by the following homological property: the $n$-preprojective algebras over $(n-1)$-representation-infinite algebras are exactly the bimodule $n$-Calabi--Yau algebras of Gorenstein parameter $1$ \cite{Kel11,MM11,HIO14,AIR15}. \\

Higher preprojective algebras and higher representation-infinite algebras have been studied extensively in the last few years. Numerous algebraic and geometric applications were developed (see, e.g., \cite{IO13, HIO14, AIR15}). In the context of non-commutative algebraic geometry, Minamoto and Mori \cite{MM11} studied the closely related notions of graded AS-regular algebras of Gorenstein parameter $1$ and (quasi)-Fano algebras. Buchweitz and Hille \cite{BH} also proved that $n$-representation-infinite algebras arise as the endomorphism algebras of certain tilting sheaves. Moreover, in the setting of higher Auslander--Reiten theory, the authors in \cite{HIMO14} generalized the notion of Geigle--Lenzing spaces \cite{GL87} and canonical algebras, which have had many applications in representation theory. There, a certain class of $n$-canonical algebras are $n$-representation-infinite. \\

When $G<SL(2, k)$ is a finite group, Reiten and Van den Bergh \cite{RVdB89} established a Morita equivalence between the skew-group algebra $k[x, y]\# G$ and the preprojective algebra over an extended Coxeter--Dynkin quiver associated to $G$ via the McKay correspondence, creating yet another bridge between representation theory and the study of quotient singularities. It is thus natural to ask whether the statement generalizes to the setting of higher Auslander--Reiten theory. 

\begin{question*}
Let $R$ be the polynomial ring in $n$ variables. Is the skew-group algebra $R\#G$ Morita equivalent to a higher preprojective algebra for any finite group $G<SL(n, k)$?
\end{question*}

Interesting applications arise when this is the case. For example, by \cite{AIR15}, we then obtain a triangle equivalence
\[
	\sCM^{\Z}(R^G)\xrightarrow{\sim}\bd(A) 
\]
between the stable category of graded Cohen--Macaulay modules over the invariant ring $R^G$ and the derived category of some finite-dimensional algebra $A$. We also have an equivalence
\[
	\sCM(R^G)\xrightarrow{\sim}\mathcal C_{n-1}(A),  
\]
where $\mathcal C_{n-1}(A)$ is the generalized $(n-1)$-cluster category (\cite{Ami09, Guo11}). In \cite{AIR15}, the authors showed that if $G$ is cyclic and satisfies an extra condition, then $R\#G$ is isomorphic to an $n$-preprojective algebra. In this paper, we prove a partial converse to this statement and generalize it to any finite subgroup of $SL(n, k)$, thus providing a negative answer to our  question. 

\begin{thma}[Theorem \ref{important} \& Corollary \ref{cor:main_section5}]
If $G$ is conjugate to a finite subgroup of $SL(n_1, k)\times SL(n_2, k)$ for some $n_1,n_2\geq 1$, then $R\#G$ is not Morita equivalent to a higher preprojective algebra.
\end{thma}

The study of this question naturally leads to the analysis of the preprojective algebra structure on the tensor product of two bimodule Calabi--Yau Koszul algebras. In this paper, we show that such algebras cannot be endowed with a grading as required for a preprojective algebra. One advantage of the Koszulity assumption is that we can describe Calabi-Yau Koszul algebras as derivation-quotient algebras $D(\omega, n-2)$, whose relations are given by differentiating a superpotential $\omega$, as proved in \cite{BSW10}. An important ingredient is then to describe the superpotential in a tensor product of two such algebras. 

\begin{thmb}[Theorem \ref{shuffle}, Theorem \ref{no_tensor} \& Corollary \ref{cor:no_tensor}]
Let $A := A^1\otimesk A^2$ be the tensor product of two bimodule $n_i$-Calabi--Yau Koszul algebras. Let $\omega_i$ be a superpotential so that $A^i\cong D(\omega_i, n_i-2)$, $i=1,2$. 

\begin{enumerate}[a)]
\item The shuffle product $\omega := \omega_1\shuffle \omega_2$ is a superpotential in $A$ and there is an isomorphism $A \cong D(\omega, n-2)$.
\item $A$ is not Morita equivalent to a higher preprojective algebra.
\end{enumerate}
\end{thmb}

Both Theorem A and Theorem Bb) rely on the fact that if we find a grading on these algebras such that they are bimodule $n$-Calabi--Yau of Gorenstein parameter $1$, then we force the degree $0$ part to be infinite-dimensional. This suggests that there could be analogues of higher hereditary algebras in the setting of infinite-dimensional algebras. \\

The objects of interest in this paper are Koszul preprojective algebras. We generalize the quiver construction of classical preprojective algebras to higher preprojective algebras over Koszul algebras. Moreover, we show that the preprojective algebra over a Koszul $n$-representation-infinite algebra has a superpotential and is given by a derivation-quotient algebra, generalizing a theorem by Keller \cite{Kel11} which states that $3$-preprojective algebras are Jacobian algebras given by a potential. This description was also shown independently by Grant and Iyama \cite{GI19} for preprojective algebras over Koszul $n$-hereditary algebras. 

\begin{thmc}[Theorem \ref{ext_construction}, Corollary \ref{construction} \& Theorem \ref{thm:der_quotient_preproj}]
Let $\Lambda = kQ/\langle M\rangle$ be a basic Koszul finite-dimensional algebra of global dimension $n$ and $\Pi:= T_{\Lambda}\Ext^n_{\Lambda}(D\Lambda, \Lambda)$. 
\begin{enumerate}[a)]
\item Let $\overline Q$ be the quiver obtained from $Q$ by adding an arrow $a_q: j\to i$ for each $k$-basis element $q$ in $e_j (\bigcap_{\mu} (kQ_1^{\otimes\mu}\otimes_{kQ_0} M \otimes_{kQ_0} kQ_1^{\otimes n-\mu-2})) e_i$ and let $\tilde M$ be a certain space of quadratic relations. Then
\[
	\Pi\cong k\overline Q/\langle M, \tilde M\rangle.
\]   
\item If $n\leq 2$ or $\Lambda$ is $n$-representation-infinite, then $\Pi$ is a derivation-quotient algebra:
\[
	\Pi\cong D(\omega, (n+1)-2), 
\]
whose relations are given by differentiating a superpotential $\omega$ of order $n+1$ by paths of length $n-1$.  
\end{enumerate}
\end{thmc}

The structure of this paper is as follows. In section 2, we give background material on higher preprojective algebras. We also define Koszul algebras and recall the notion of a superpotential, described in \cite{BSW10}, that plays an important role in the grading structure of the algebras we consider. Section 3 is devoted to proving Theorem C. We give an explicit construction of the quiver of a basic $(n+1)$-preprojective algebra $\Pi(\Lambda)$ over a Koszul algebra $\Lambda$, generalizing the quiver construction in the classical case. We then prove that if $\Lambda$ is $n$-representation-infinite, then such a preprojective algebra is a derivation-quotient algebra with quadratic relations. In section 4, we show Theorem B. We first prove in section 4.1 that the shuffle product of the superpotentials in two Koszul bimodule Calabi--Yau algebras is a superpotential in the tensor product of these two algebras. We then use it to prove in section 4.2 that the tensor product of two Koszul bimodule Calabi--Yau algebras cannot have the structure of a preprojective algebra. In section 5, we extend these results in the context of skew-group algebras, proving Theorem A, which generalizes a partial converse to a theorem in \cite{AIR15}. In section 5.1, we directly use the results from section 4 to give a class of groups for which the skew-group algebra is not Morita equivalent to a preprojective algebra. In section 5.2, we enlarge this class to include any group which embeds in $SL(n_1, k)\times SL(n_2, k)$.      

\subsection*{Notation}
We let $k$ be an algebraically closed field of characteristic $0$. We denote by $D := \Hom_k (-, k)$ the $k$-dual. Let $A$ be a (graded) $k$-algebra. We denote by $A^e:= A\otimesk A^{op}$ the enveloping algebra. 
We denote by $\mathsf{Mod}\, A$ the category of left $A$-modules, $\mathsf{Bimod}\, A\cong \Mod A^e$ the category of $A$-bimodules, $\mathsf{Gr}\, A$ the category of graded left $A$-modules and $\mathsf{Grproj}\, A$ the category of graded projective left $A$-modules. If the names are written with a lowercase, then they denote the respective full subcategories of finitely generated modules. Moreover, we denote by $\DD(-)$ and $\DD^b(-)$ the derived and bounded derived categories. The category $\mathsf C^b(-)$ is the category of bounded complexes. Finally, $\mathsf{per}\, A$ is the thick subcategory of $\DD(\Mod A)$ generated by $A$. 


\section*{Acknowledgement}

The results described in this article were obtained while the author was a Ph.D. student at the University of Toronto. The author wishes to thank his supervisor, Ragnar-Olaf Buchweitz, for the numerous insights, suggestions, and the constant support throughout his Ph.D. The author also thanks Osamu Iyama for discussing some of the results during his stay in Toronto. Finally, the author thanks his colleague Vincent G\'elinas for his helpful comments on the first drafts of this paper, as well as Steffen Oppermann, Mads Hustad Sand\o y and Torkil Stai for their valuable inputs on later drafts. 


\section{Preliminaries}


\subsection{Higher preprojective algebras}
Let $\Lambda$ be a finite-dimensional algebra of global dimension $n$. Let 
\[
	\Se:= D\Lambda\Lotimes_{\Lambda}- : \bd(\mathsf{mod}\, \Lambda)\to \bd(\mathsf{mod}\, \Lambda)
\]
be the Serre functor and denote by $\Se_n$ the composition $\Se_n:= \Se\circ [-n]$. The inverse of the Serre functor is given by 
\[
	\Se^{-1}=\RHom_{\Lambda}(D\Lambda, -): \bd(\mathsf{mod}\, \Lambda)\to \bd(\mathsf{mod}\, \Lambda).
\]

\begin{definition}[{\cite[Definition 2.7]{HIO14}}]
We say that $\Lambda$ is \emph{$n$-representation-infinite} if 
\[
	\Se_n^{-\ell}(\Lambda)\in\mathsf{mod}\,\Lambda
\]
for any $\ell\geq 0$. 
\end{definition}

By definition, $\Se_n^{-\ell}(\Lambda)$ is a complex, so the condition means that it has only cohomology in degree $0$. This cohomology is then necessarily equal to $\Ext_{\Lambda}^n(D\Lambda, \Lambda)^{\otimes_{\Lambda} \ell}$.

\begin{definition}[{\cite[Definition 2.11]{IO13}}]
Let $\Lambda$ be an $n$-representation-infinite algebra. The \emph{$(n+1)$-preprojective algebra} of $\Lambda$ is defined as the tensor algebra
\[
	\Pi(\Lambda) = T_{\Lambda}\Ext_{\Lambda}^n(D\Lambda, \Lambda)\cong \bigoplus_{\ell\geq 0} \Se_n^{-\ell}(\Lambda).
\]
\end{definition}

These notions generalize the concepts of hereditary representation-infinite algebras and of preprojective algebras to algebras of higher global dimension. 

\begin{definition}[{\cite[Definition 3.2]{AIR15}}]
\label{def_bim_cy}
Let $A = \bigoplus_{\ell\geq 0} A_{\ell}$ be a positively graded $k$-algebra and $a\in\Z$. We say that $A$ is \emph{(bimodule) $n$-Calabi--Yau of Gorenstein parameter $a$} if $A\in\mathsf{per}\, A^e$ and there exists a graded projective $A^e$-module resolution $P_{\bullet}$ of $A$ and an isomorphism 
\begin{equation}
\label{self_dual}
	P_{\bullet}\cong P_{\bullet}^{\vee}[n](-a)\quad\text{in } \mathsf C^b(\mathsf{grproj}\,A^e),
\end{equation}
where $(-)^{\vee} = \Hom_{A^e}(-, A^e)$. Note that this implies that $\mathsf{gl.dim}\, A = n$ (\cite[Proposition 2.4]{AIR15}). We call an algebra \emph{(bimodule) $n$-Calabi--Yau} if it satisfies the previous homological property, forgetting the grading. 
\end{definition}

Calabi--Yau algebras give rise to Calabi--Yau categories as follows. 

\begin{proposition}[{\cite[Proposition 3.2.4]{Gin06}; \cite[Lemma 4.1]{Kel08}; \cite[Proposition 2.4]{AIR15}}]
\label{prop:cy_duality}
Let $A$ be a bimodule $n$-Calabi--Yau algebra. For any $X\in \DD(A)$ and $Y\in \DD^{\mathsf{fd}} (A)$, the derived category of complexes with finite-dimensional cohomology, there is a functorial isomorphism
\[	
	\Hom_{\DD(A)}(X, Y )\cong D\Hom_{\DD(A)}(Y, X[n]).
\]
Therefore, $\DD^{\mathsf{fd}} (A)$ is an $n$-Calabi--Yau triangulated category.
\end{proposition}

\begin{theorem}[{\cite[Theorem 4.8]{Kel11}; \cite[Corollary 4.13]{MM11}; \cite[Theorem 4.36]{HIO14}; \cite[Theorem 3.4]{AIR15}}]
\label{structure}
There is a one-to-one correspondence 
\begin{align*} 
\left\{\text{$n$-representation-infinite algebras $\Lambda$}\right\}/\cong&\xleftrightarrow{1:1}\left\{ \text{\parbox{6.8cm}{bimodule $(n+1)$-Calabi--Yau algebras $A$ of Gorenstein parameter $1$ such that $\dim_k A_0 <\infty$}}\right\}/\cong\\
\Lambda \qquad\qquad\qquad\qquad\quad&\longmapsto \qquad\qquad\qquad\qquad\quad \Pi(\Lambda) \\ 
 A_0 \qquad\qquad\qquad\qquad\;\;&\longmapsfrom \qquad\qquad\qquad\qquad\quad A
\end{align*}

\end{theorem}

\begin{remark}
\label{rem:locally_finite}
As mentioned in \cite[Remark 3.2]{AIR15}, if $A = \bigoplus_{\ell\geq 0} A_{\ell}$ is locally finite, that is, $\dim_k A_{\ell}<\infty$ for all $\ell$, then $A$ is bimodule $(n+1)$-Calabi--Yau of Gorenstein parameter $1$ if and only if $A\in \mathsf{per} \, A^e$ and there exists an isomorphism
\[
	\RHom_{A^e} (A, A^e)[n+1](-1)\cong A\quad \text{in } \DD(\mathsf{Gr}\, A^e).
\]
The minimal projective $A$-bimodule resolution $P_{\bullet}$ of $A$ then satisfies $(\ref{self_dual})$. The (ungraded) bimodule Calabi--Yau algebras are also characterized by those properties, if we remove the graded part. If $A$ is bimodule $(n+1)$-Calabi--Yau of Gorenstein parameter $1$ and $A_0$ is finite-dimensional, then $A_{\ell}\cong \Ext^n_{A_0}(DA_0, A_0)^{\otimes_{A_0}\ell}$ is finite-dimensional for any $\ell$, and so $A$ is locally finite.\\

By our definition, the degree $0$ part of a preprojective algebra is a higher representation-infinite algebra, and so is finite-dimensional. This property is essential for many developments on preprojective algebras, for example for the Serre functor to be well-defined. However, when we use the terminology  \emph{bimodule Calabi--Yau algebra of Gorenstein parameter $1$}, we do not necessarily require the degree $0$ part to be finite-dimensional. \\

Note that the only finite-dimensional $n$-Calabi--Yau algebras $A$ are semisimple. In fact, by Proposition \ref{prop:cy_duality}, the Calabi--Yau property implies 
\[
	\Hom_A(X, Y)\cong D\Ext_A^n(Y, X),
\]
 for any finite-dimensional modules $X$ and $Y$. By considering the case where $Y$ is a projective module and $X$ is a module such that there exists a non-zero morphism $X\to Y$, we see that this can only hold if $n=0$. In this paper, we exclude semisimple Calabi--Yau algebras, as we are mainly interested in higher preprojective algebra structures.       
\end{remark}


\subsection{Koszul algebras}
\label{Koszul}

We describe Koszul algebras, following mostly \cite{BGS96}. Throughout this subsection, we let $A = \bigoplus_{i\geq 0} A_i$ be a positively graded $k$-algebra such that $S: = A_0$ is a finite-dimensional semisimple algebra. All tensor products are over $S$.

\begin{definition}
The algebra $A$ is \emph{Koszul} if $S$, considered as a left graded $A$-module, admits a graded projective resolution
\[
	\cdots \to Q_2 \to Q_1 \to Q_0\to S\to 0
\]
such that $Q_{\ell}$ is generated in degree $\ell$.
\end{definition}

It is well-known that any Koszul algebra is \emph{quadratic}, that is, there is a graded isomorphism
\[
	A\cong T_S V/\langle M \rangle,
\]
where $V$ is a $S$-bimodule placed in degree $1$ and $M$ is a $S$-bimodule such that $M\subset V\otimes V$. Here $\langle M \rangle$ denotes the $2$-sided ideal generated by $M$.

\begin{definition}
\label{koszul_complex}
The \emph{Koszul complex} is described as
\[
	 \cdots\to Q_n\to Q_{n-1}\to\cdots\to Q_0\to 0,
\]
where each $Q_{\ell}$ is given by 
\[
	A\otimes \bigcap_{\mu} (V^{\otimes\mu}\otimes M \otimes V^{\otimes \ell-\mu-2})\subset A\otimes V^{\otimes \ell} 
\]
and the differentials are the restrictions of the maps 
\[
	A\otimes V^{\otimes \ell}\to A\otimes V^{\otimes (\ell-1)},
\]
given by 
\[
	a\otimes v_1\otimes\cdots\otimes v_{\ell}\mapsto av_1\otimes v_2\otimes\cdots\otimes v_{\ell}.
\]
We write 
\[
 K_{\ell}:=\bigcap_{\mu} (V^{\otimes\mu}\otimes M \otimes V^{\otimes \ell-\mu-2}).
\]
\end{definition}

\begin{theorem}[{\cite[Theorem 2.6.1]{BGS96}}]
Let $A = T_S V/\langle M \rangle$ be a quadratic algebra. Then $A$ is Koszul if and only if the Koszul complex is a projective left $A$-module resolution of $S$.
\end{theorem}

In this case, the Koszul resolution is a subresolution of the \emph{Bar resolution} $\mathbb B(A,S)$ of $S$:
\begin{equation}
\label{bar_res}
	\mathbb B(A, S): = \cdots\to A\otimes A_+^{\otimes n}\to A\otimes A_+^{\otimes (n-1)}\to\cdots \to A,
\end{equation}
where $A_+:=\bigoplus_{i\geq 1}A_i$, with differentials given by
\begin{equation*}
\begin{split}
	1\otimes (a_1\otimes\cdots\otimes &a_n) \mapsto \\
	a_1\otimes &(a_2\otimes\cdots\otimes a_n)+\sum_{1\leq \ell\leq n-1} (-1)^{\ell} 1\otimes (a_1\otimes\cdots a_{\ell-1}\otimes a_{\ell}a_{\ell+1}\otimes a_{\ell+2}\otimes\cdots\otimes a_n).
\end{split}
\end{equation*}
Combining the discussion after Definition 2.8.1 with Theorem 2.10.1 in \cite{BGS96}, we deduce that
\begin{equation}
\label{tor_kl}
	K_{\ell} \cong \Tor^A_{\ell}(S,S)
\end{equation}
as $S$-bimodules. The Koszul resolution also gives rise to a projective $A$-bimodule resolution of $A$:
\[
	\cdots\to P_n\to P_{n-1}\to\cdots\to P_0\to A\to 0,
\]
where each $P_{\ell}$ is given by 
\[
	A\otimes \bigcap_{\mu} (V^{\otimes\mu}\otimes M \otimes V^{\otimes \ell-\mu-2})\otimes A\subset A\otimes V^{\otimes \ell} \otimes A
\]
and the differentials are the restrictions of the maps 
\[
	A\otimes V^{\otimes \ell}\otimes A\to A\otimes V^{\otimes (\ell-1)}\otimes A,
\]
given by
\[
	a\otimes v_1\otimes\cdots\otimes v_{\ell}\otimes a'\mapsto av_1\otimes v_2\otimes\cdots\otimes v_{\ell}\otimes a' + (-1)^{\ell} a\otimes v_1\otimes\cdots\otimes v_{\ell-1}\otimes v_{\ell} a'.
\]

\begin{remark}
\label{grading_type}
In this paper, we deal in general with two different gradings. The first one is the natural grading coming from the tensor algebra structure 
\[
	A = T_S V/\langle M \rangle.
\]
We consider Koszul algebras with respect to this grading and call it the \emph{Koszul grading}. If $A$ is Calabi--Yau of Gorenstein parameter $1$, then we usually have a second grading on $A$ giving this property. In particular, when $A$ is a preprojective algebra, this grading is given by the natural grading structure on the tensor algebra
\[
	A = T_{\Lambda} \Ext^n_{\Lambda}(D\Lambda, \Lambda).
\]
We call this grading the \emph{preprojective grading}. 
\end{remark}


\subsection{Superpotentials}

Let $A = T_S V/\langle M \rangle$ be a $k$-algebra, where $S$ is a semisimple finite-dimensional $k$-algebra, and $V$ and $M$ are $S$-bimodules. We describe the notion of a superpotential, following \cite{BSW10}. In this subsection, all tensor products are over $S$. \\

Let $\tau\in\mathfrak{S}_{\ell}$ be an element of the symmetric group acting on 
\[
	v_{\ell}\otimes v_{\ell-1}\otimes \cdots\otimes v_1\in V^{\otimes \ell}
\] 
by permuting the indices, that is, 
\begin{equation*}
\label{sym_action}
	\tau(v_{\ell}\otimes v_{\ell-1}\otimes \cdots\otimes v_1):= v_{\tau^{-1}(\ell)}\otimes v_{\tau^{-1}(\ell-1)}\otimes\cdots\otimes v_{\tau^{-1}(1)}.
\end{equation*}
This action extends linearly. In this paper, we shall use the notation
\begin{equation}
\label{sigma_n}
	\sigma_n := (1\,\, 2\,\,\cdots \,\,n-1\,\, n)\in \mathfrak{S}_{n}. 
\end{equation} 

\begin{definition}[{\cite[Section 2.2]{BSW10}}]
A \emph{superpotential} of degree $n$ is an element $\omega$ of degree $n$ in $T_S V$ that commutes with the $S$-action:
\[
	s\omega = \omega s\quad\forall s\in S,
\] 
and is \emph{super-cyclically symmetric}, that is, $\sigma_n(\omega) = (-1)^{n-1}\omega$.
\end{definition}

\begin{lemma}
\label{super_cycl}
An element $\omega$ of degree $n$ in $T_S V$ is super-cyclically symmetric if and only if it is given as a linear combination of elements of the form 
\[
	\sum_{i = 0}^{n-1} (-1)^{i(n-1)} \sigma_n^i(v_n\otimes\cdots\otimes v_1),
\]
for some $v_1, \ldots, v_n\in V$. 
\end{lemma}

\begin{proof}
We have that 
\begin{equation*}
\begin{split}
	\sigma_n\left(\sum_{i = 0}^{n-1} (-1)^{i(n-1)} \sigma_n^i(v_n\otimes\cdots\otimes v_1)\right) &= \sum_{i = 0}^{n-1} (-1)^{i(n-1)} \sigma_n^{i+1}(v_n\otimes\cdots\otimes v_1)\\
	& = (-1)^{n-1}\sum_{j = 1}^{n} (-1)^{j(n-1)} \sigma_n^{j}(v_n\otimes\cdots\otimes v_1),
\end{split}
\end{equation*}
so this element is super-cyclically symmetric. Conversely, if $v_n\otimes\cdots\otimes v_1$ is a summand of a super-cyclically symmetric element $\omega$, then $(-1)^{n-1}\sigma_n(v_n\otimes\cdots\otimes v_1)$ must also be. Repeating this, we see that 
\[
	\sum_{i = 0}^{n-1} (-1)^{i(n-1)} \sigma_n^i(v_n\otimes\cdots\otimes v_1)
\]
must be a summand of $\omega$. 
\end{proof}

Let $W_{n-\ell}$ be the $S$-bimodule generated by elements of the form $\delta_p \omega$, where $p\in V^{\otimes \ell}$ is a $k$-basis element and  $\delta_p: V^{\otimes n}\to V^{\otimes n-\ell}$ is a linear map defined on a $k$-basis element $v\in V^{\otimes n}$ as
\[
	\delta_p v : = \left\{
	\begin{array}{ll}
		r  & \mbox{if } v = p\otimes r, \\
		0 & \mbox{else. }
	\end{array}
\right.
\]

\begin{definition}[{\cite[Definition 2.1]{BSW10}}]
\label{def:der_quotient}
The \emph{derivation-quotient} algebra of $\omega$ of order $\ell$ is defined as
\[
	D(\omega,\ell):= T_S V/\langle W_{n-\ell}\rangle.
\]  
\end{definition}

Consider the complex 
\begin{equation}\label{eq:w}
	W_{\bullet}: \,\,0\to A\otimes W_n\otimes A\xrightarrow{d} A\otimes W_{n-1}\otimes A\xrightarrow{d}\cdots\xrightarrow{d}A\otimes W_1\otimes A\xrightarrow{d} A\otimes W_0\otimes A\to 0,
\end{equation}
where the differential $d: A\otimes W_{\ell}\otimes A\to A\otimes W_{\ell-1}\otimes A$ is given by 
\[
	d(a\otimes v_{\ell}\otimes\cdots\otimes v_1\otimes a') 	= \varepsilon_{\ell}(av_{\ell}\otimes v_{\ell-1}\otimes\cdots\otimes v_1\otimes a' + (-1)^{\ell} a\otimes v_{\ell}\otimes\cdots\otimes v_2\otimes v_1 a'),
\]
where 
\[
	\varepsilon_{\ell}:=
\left\{
	\begin{array}{ll}
		(-1)^{\ell(n-\ell)}  & \mbox{if } \ell< (n+1)/2 \\
		1 & \mbox{else.}
	\end{array}
\right.
\]
For any superpotential, this complex is self-dual, leading to a Calabi--Yau property. Moreover, it is a subcomplex of the Koszul complex for $D(\omega, n-2)$.

\begin{theorem}[{\cite[Theorem 6.2]{BSW10}}]
\label{BSW}
Let $A = T_S V/ \langle M \rangle$. Then $A$ is an $n$-Calabi--Yau Koszul algebra if and only if $A\cong D(\omega,n-2)$, for some superpotential $\omega$ of degree $n$, and $W_{\bullet}$ is exact in positive degree with $H^0(W_{\bullet}) = A$. In this case, $W_{\bullet}$ is the Koszul resolution of $A$ and is self-dual.  
\end{theorem}

\begin{remark}
\label{sp_gen}
In particular, if  $A = T_S V/ \langle M \rangle$ is an $n$-Calabi--Yau Koszul algebra, then the superpotential $\omega$ is a $S$-bimodule generator of $K_n$. In fact, there is an isomorphism $W_n\cong K_n$, so the claim follows by definition of $W_n$. Similarly, all $K_{\ell}$ are generated as $S$-bimodules by elements of the form $\delta_p\omega$, where $p$ is an element of degree $\ell$ in $T_S V$.  
\end{remark}


\section{Quiver construction of preprojective algebras}

The goal of this section is to generalize the quiver construction of $2$-preprojective algebras to higher preprojective algebras over basic Koszul algebras. In the classical case, we have the following result.

\begin{theorem}[{\cite[Theorem A]{Rin98}; \cite[Theorem 3.1]{CB99}}]
Let $Q = (Q_0, Q_1)$ be a quiver with no oriented cycle and consider its path algebra $kQ$. For each arrow $a:i\to j\in Q_1$, define $a^*:j\to i$. Consider $\overline Q = (Q_0, \overline Q_1)$, where $\overline Q_1$ contains all the arrows $a$ of $Q_1$ as well as the new arrows $a^*$. Then, the preprojective algebra of $kQ$ is given by
\[
	\Pi(kQ) \cong k\overline Q/\langle\sum_{a\in Q_1} [a, a^*]\rangle,
\]
where $[a, a^*]=aa^*-a^*a$.
\end{theorem}

In this section, we let $\Lambda = kQ/\langle M\rangle = T_S V /\langle M \rangle$ be a basic Koszul finite-dimensional algebra of global dimension $n\geq 1$, where $Q$ is a quiver, $S=kQ_0$ is a finite-dimensional semisimple $k$-algebra, $V=kQ_1$ a $S$-bimodule and $M\subset V\otimes_S V$. Let 
\[
	\Pi = T_{\Lambda} \Ext^n_{\Lambda} (D\Lambda, \Lambda) \cong T_{\Lambda} \Ext^n_{\Lambda^e}(\Lambda, \Lambda^e).
\]

The following is inspired by techniques employed in \cite{CB99}. By \cite[Section 1.5]{Hap89}, a minimal bimodule resolution of $\Lambda$ has the form 
\[
	0\to R_n\to R_{n-1}\to \ldots \to R_1\to R_0 \to \Lambda\to 0, 
\]
where 
\[
	R_{\ell} = \bigoplus_{i,j} (\Lambda e_j\otimes e_i\Lambda)^{\dim \Ext^{\ell}_{\Lambda}(S(i), S(j))}, 
\]
for $0\leq \ell \leq n$. Since $\Lambda$ is Koszul, we can invoke (\ref{tor_kl}) to rewrite its Koszul bimodule resolution as
\begin{equation}
\label{eq:koszul_complex}
0\to \bigoplus_{\substack{q: i\to j\\ q\in \mathcal B(K_n)}}\Lambda e_j\otimesk e_i\Lambda\xrightarrow{f_n}\bigoplus_{\substack{p: i\to j\\ p\in \mathcal B(K_{n-1})}}\Lambda e_j\otimesk e_i\Lambda\xrightarrow{f_{n-1}} \cdots \xrightarrow{f_1}\Lambda\otimes_S\Lambda\to\Lambda\to 0,
\end{equation}
where $K_{\ell}$ is the $S$-bimodule introduced in Definition \ref{koszul_complex}, for $1\leq \ell \leq n$, and the direct sums are indexed by a set $\mathcal B(K_{\ell})$ of $k$-basis elements.\\

From now on in this section, if $q = v_{\ell}\otimes v_{\ell-1}\otimes\cdots\otimes v_2\otimes v_1$ is a path from $i$ to $j$, then we let $\delta^{\mathcal L} (q):= v_{\ell-1}\otimes\cdots\otimes v_1$, which is a path from $i$ to some vertex $j'$, and also let $\delta^{\mathcal R}  (q):= v_{\ell}\otimes\cdots \otimes v_2$, which is a path from some vertex $ i'$ to $j$. Then $f_{\ell}$ is defined as follows: 
\[
	f_{\ell}((e_j\otimes e_i)_q) = (v_{\ell} e_{j'}\otimes e_i)_{\delta^{\mathcal L} (q)} + (-1)^{\ell}(e_j\otimes e_{i'} v_1)_{\delta^{\mathcal R}  (q)}.
\]
Here, we use the notation $(-\otimes -)_p$ to denote an element in the $p$-th component in $\bigoplus_p \Lambda e_i\otimes_k e_j \Lambda$. We extend this definition linearly. This corresponds to the usual differential in the Koszul resolution.\\

Applying $\Hom_{\Lambda^e}(-, \Lambda^e)$, we obtain a complex ending as follows: 
\[
	\Hom_{\Lambda^e}(\bigoplus_{\substack{p: i\to j\\ p\in \mathcal B(K_{n-1})}}\Lambda e_j\otimesk e_i\Lambda, \Lambda^e)\xrightarrow{\tilde f_n}\Hom_{\Lambda^e}(\bigoplus_{\substack{q: i\to j\\ q\in \mathcal B(K_n)}}\Lambda e_j\otimesk e_i\Lambda, \Lambda^e)\to 0,
\]
where $\tilde f_n(\phi) := \phi\circ f_n$. In general, if $\Omega$ is a $\Lambda$-bimodule, then 
\begin{align*}
	\kappa:\Hom_{\Lambda^e}(\Lambda e_j\otimesk e_i\Lambda, \Omega) &\cong e_j\Omega e_i\\
	\phi\quad\quad\quad\quad &\mapsto \phi(e_j\otimes e_i)\\
	(\phi_m: be_j\otimes e_i b'\mapsto be_jme_ib') &\mapsfrom e_j m e_i
\end{align*}

Thus, taking $\Omega = \Lambda^e$, we obtain an isomorphism 
\[
	g_{\ell}: \Hom_{\Lambda^e}(\bigoplus_{\substack{p: i\to j\\ p\in \mathcal B(K_{\ell})}}\Lambda e_j\otimesk e_i\Lambda, \Lambda^e)\overset{\kappa_{\ell}}{\cong} \bigoplus_{\substack{p: i\to j\\ p\in \mathcal B(K_{\ell})}}e_j\Lambda\otimesk\Lambda e_i\overset{\rho_{\ell}}{\cong} \bigoplus_{\substack{p: i\to j\\ p\in \mathcal B(K_{\ell})}}\Lambda e_i\otimesk e_j\Lambda,
\]
where $\rho_{\ell}$ is the natural isomorphism that swaps the order of the terms in the tensor product. \\

Therefore, 
\[
	\Ext_{\Lambda^e}^n(\Lambda, \Lambda^e)\cong \bigoplus_{\substack{q: i\to j \\ q\in \mathcal B(K_n)}} \Lambda e_i\otimesk e_j \Lambda/ \Image (g_n\circ\tilde f_n).
\]

We now describe $ \Image(g_n\circ\tilde f_n)$. Let $q: i\to j$ be a $k$-basis element in $K_n\subset V^{\otimesS n}$, say $q = v_n\otimes \cdots\otimes v_1$. Then 
\begin{align*}
	f_n((e_j\otimes e_i)_q) &= (v_ne_{j'}\otimes e_i)_{\delta^{\mathcal L} (q)} + (-1)^n (e_j\otimes e_{i'}v_1)_{\delta^{\mathcal R} (q)}\\
	&= (e_j v_n \otimes e_i)_{\delta^{\mathcal L} (q)} + (-1)^n (e_j\otimes v_1 e_i)_{\delta^{\mathcal R} (q)}.\\
\end{align*}

Note that the space 
\[	
	 \Hom_{\Lambda^e}(\bigoplus_{\substack{p: i\to j\\ p\in \mathcal B(K_{n-1})}}\Lambda e_j\otimesk e_i\Lambda, \Lambda^e)
\]
is generated as a $\Lambda$-bimodule by morphisms of the form 
\[
	\phi_p: (e_j\otimes e_i)_{p'}\mapsto
		\begin{cases}
			 e_j\otimes e_i  &\quad\text{if } p = p',\\
			 0 & \quad \text{else}.
		\end{cases}
\]
for $p, p'\in \mathcal B(K_{n-1})$. Thus, 
\begin{align*}
	\kappa_n(\tilde f_n(\phi_p))((e_j\otimes e_i)_q) &= \phi_p\circ f_n((e_j\otimes e_i)_q)\\
	& = \begin{cases}
		(e_jv_n\otimes e_i)_{q} &\quad\text{if } p = \delta^{\mathcal L} (q),\\
		(-1)^n(e_j\otimes v_1e_i)_{q}&\quad\text{if } p = \delta^{\mathcal R} (q),\\
		0 &\quad\text{else}.\\
	\end{cases}
\end{align*}
Finally, applying $\rho_n$, we have  
\begin{equation}
\label{eq:diff_f}
	g_n(\tilde f_n(\phi_p))((e_j\otimes e_i)_q)  = \begin{cases}
		(e_i\otimes e_jv_n)_{q} &\quad\text{if } p = \delta^{\mathcal L} (q),\\
		(-1)^n(v_1e_i\otimes e_j)_{q}&\quad\text{if } p = \delta^{\mathcal R} (q),\\
		0 &\quad\text{else}.\\
	\end{cases}
\end{equation}
If $p\in V^{\ell}$, for some $0\leq \ell \leq n$, we define the following linear morphisms on elements $q\in \mathcal B(V^{n})$ as 
\[
	\delta^{\mathcal L}_p (q) := \left\{
	\begin{array}{ll}
		a  & \mbox{if } q = p\otimes a \\
		0 & \mbox{else} ,
	\end{array}
\right.
\]
and 
\[
	\delta^{\mathcal R}_p (q) := \left\{
	\begin{array}{ll}
		a  & \mbox{if } q = a\otimes p \\
		0 & \mbox{else}.
	\end{array}
\right.
\]
We conclude that $ \Image (g_n\circ\tilde f_n)$ is generated as a $\Lambda$-bimodule by the following set in ${	\bigoplus_{\substack{q: i\to j \\ q\in K_n}} \Lambda e_i\otimesk e_j \Lambda}$:
\[
      I:= \left\{ \sum_{\substack{q:i\to j \\ q\in \mathcal B(K_n)}} (e_i\otimes e_j \delta^{\mathcal R}_p (q))_{q} + (-1)^n\sum_{\substack{q:i\to j \\ q\in \mathcal B(K_n)}}  (\delta^{\mathcal L}_p (q) e_i\otimes e_j)_{q}\quad | \quad p\in \mathcal B(K_{n-1}) \right\}.
\]
We thus proved the following theorem.

\begin{theorem}
\label{ext_construction}
With the setting above, there is an isomorphism of $\Lambda$-bimodules
\[
	\Ext_{\Lambda^e}^n(\Lambda, \Lambda^e)\cong \bigoplus_{\substack{q: i\to j \\ q\in \mathcal B(K_n)}} \Lambda e_i\otimesk e_j \Lambda/ \langle I \rangle.
\]
\end{theorem}

We can now describe explicitely the quiver of the preprojective algebra over a basic Koszul $n$-representation-infinite algebra. 

\begin{corollary}
\label{construction}
Assume that $\Pi$ and $\Lambda$ are as above. Let $\overline Q$ be the quiver obtained by adding an arrow $a_q: j\to i$ to the quiver $Q$ of $\Lambda$ for each $k$-basis element $q: i\to j\in \mathcal B(K_n)$. Let $\overline M$ be the union of $M$ with the set $\tilde M$ of quadratic relations given by
\[
	\tilde M:=  \left\{\sum_{q\in \mathcal B(K_n)}a_q \delta^{\mathcal R}_p (q)+ (-1)^n \sum_{q\in \mathcal B(K_n)}\delta_p^{\mathcal L} (q) a_q \quad | \quad p\in \mathcal B(K_{n-1}) \right\}.
\] 
There is an isomorphism of algebras
\[
	\Pi\cong k\overline Q/\langle \overline M\rangle.
\]

\end{corollary}

\begin{proof}
We first consider the morphism
\[
	\varphi: T_{\Lambda}  \left(\bigoplus_{\substack{q: i\to j \\ q\in \mathcal B(K_n)}} \Lambda e_i\otimesk e_j\Lambda\right)\rightarrow k\overline Q/\langle M\rangle
\]
by setting $\varphi(\lambda) := \lambda$ if $\lambda\in\Lambda = kQ/\langle M\rangle$ and $\varphi((e_i\otimes e_j)_q) := a_q$. Note that the codomain only has the relations $\langle M \rangle$ coming from $\Lambda$. This extends naturally to an algebra morphism by defining
\[
	\varphi((\lambda_1 e_i\otimes e_j \lambda_1')_{q_1}\otimes_{\Lambda} (\lambda_2 e_{\ell}\otimes e_m \lambda_2')_{q_2}) := \lambda_1 a_{q_1} \lambda_1'\lambda_2 a_{q_2}\lambda_2', 
\] 
and we also extend linearly. This is an isomorphism. In fact, it is injective since $\langle M\rangle \subset\Lambda$, and the map is an isomorphism on $\Lambda$. It is surjective because every arrow in $\overline Q_1$ has a preimage. \\

Applying $\varphi$ to the set $I$, we see that 
\[
      \varphi(I) = \tilde M. 
\]
Therefore, 
\[
	T_{\Lambda}  \left(\bigoplus_{\substack{q: i\to j \\ q\in K_n}} \Lambda e_i\otimesk e_j \Lambda/ \langle I \rangle\right)\cong k\overline Q/\langle \overline M\rangle.
\]
By Theorem \ref{ext_construction}, they are isomorphic to $\Pi$. 
\end{proof}

\begin{example}
Let $\Lambda$ be the $2$-Beilinson algebra, that is, the Koszul $2$-representation-infinite algebra given by the path algebra of the following quiver and relations: 
\begin{center}
\begin{tikzpicture}
 \tikzstyle{every node}=[draw,circle,fill=black,minimum size=3pt,
                            inner sep=0pt]
       \draw [decoration={markings,mark=at position 1 with {\arrow[scale=2]{>}}},
    postaction={decorate}, shorten >=0.4pt] (0,0) -- (1,0);  
    \draw [decoration={markings,mark=at position 1 with {\arrow[scale=2]{>}}},
    postaction={decorate}, shorten >=0.4pt] (0,0.5) -- (1,0.5);  
      \draw [decoration={markings,mark=at position 1 with {\arrow[scale=2]{>}}},
    postaction={decorate}, shorten >=0.4pt] (0,-0.5) -- (1,-0.5);
          \draw [decoration={markings,mark=at position 1 with {\arrow[scale=2]{>}}},
    postaction={decorate}, shorten >=0.4pt] (1.5,0) -- (2.5,0);  
    \draw [decoration={markings,mark=at position 1 with {\arrow[scale=2]{>}}},
    postaction={decorate}, shorten >=0.4pt] (1.5,0.5) -- (2.5,0.5);  
          \draw [decoration={markings,mark=at position 1 with {\arrow[scale=2]{>}}},
    postaction={decorate}, shorten >=0.4pt] (1.5,-0.5) -- (2.5,-0.5);
	 \tikzstyle{every node}=[draw, circle, fill = white, draw = none, minimum size=1pt,
                            inner sep=0pt]      
       \draw (0.5,0.75) node {$a$};
       \draw (0.5,0.25) node {$b$};
       \draw (0.5, -0.25) node {$c$}; 
       \draw (2, 0.75) node {$d$};
       \draw (2, 0.25) node {$e$ };
       \draw (2, -0.25) node {$f$};
       \draw (-0.25, 0) node{$1$};
     \draw (1.25, 0) node{$2$};
     \draw(2.75, 0) node{$3$};
\end{tikzpicture}
\end{center}
\[
	q_1:= db-ea =0,\quad q_2:=fa-dc = 0,\quad q_3:=ec-fb = 0.
\]
The module $K_2$ is generated by the relations, thus the preprojective algebra $\Pi(\Lambda)$ has three more arrows $a_{q_i}: 3\to 1$, $i=1,2,3$. Moreover, the arrows form a $k$-basis in $K_1$. Let $p = a$. Then
\[
	\sum_{q\in \mathcal B(K_2)}a_q \delta^{\mathcal R}_a (q)+ (-1)^2 \sum_{q\in \mathcal B(K_2)}\delta_a^{\mathcal L} (q) a_q = a_{q_1}(-e) + a_{q_2}f + a_{q_3}\cdot 0 + 0\cdot a_{q_1} + 0\cdot a_{q_2} + 0\cdot a_{q_3},  
\]
so $a_{q_2} f - a_{q_1}e=0$ is a new relation in $\Pi$.  Doing so for every arrow, we obtain that the new relations are given by 
\[
	a_{q_2}f- a_{q_1}e= 0, \quad  a_{q_1}d - a_{q_3}f = 0, \quad a_{q_3}e - a_{q_2}d = 0\quad
\]
and
\[
	ba_{q_1} - fa_{q_3} =0,\quad ca_{q_3} - aa_{q_1} = 0,\quad aa_{q_2} - ba_{q_3} =0. 
\]
The quiver of the preprojective algebra is given by

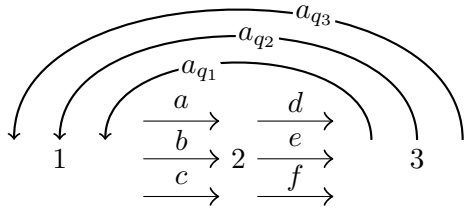
\begin{center}
\begin{tikzpicture}
 \tikzstyle{every node}=[draw,circle,fill=black,minimum size=5pt,
                            inner sep=0pt]
       \draw [decoration={markings,mark=at position 1 with {\arrow[scale=2]{>}}},
    postaction={decorate}, shorten >=0.4pt] (0,0) -- (1,0);  
    \draw [decoration={markings,mark=at position 1 with {\arrow[scale=2]{>}}},
    postaction={decorate}, shorten >=0.4pt] (0,0.5) -- (1,0.5);  
      \draw [decoration={markings,mark=at position 1 with {\arrow[scale=2]{>}}},
    postaction={decorate}, shorten >=0.4pt] (0,-0.5) -- (1,-0.5);
          \draw [decoration={markings,mark=at position 1 with {\arrow[scale=2]{>}}},
    postaction={decorate}, shorten >=0.4pt] (1.5,0) -- (2.5,0);  
    \draw [decoration={markings,mark=at position 1 with {\arrow[scale=2]{>}}},
    postaction={decorate}, shorten >=0.4pt] (1.5,0.5) -- (2.5,0.5);  
          \draw [decoration={markings,mark=at position 1 with {\arrow[scale=2]{>}}},
    postaction={decorate}, shorten >=0.4pt] (1.5,-0.5) -- (2.5,-0.5);
       \draw [->, very thick] (3,0.25) to [out=90, in =90] (-0.5,0.25);
      \draw [->, very thick] (3.6,0.25)  to [out=90,in= 90] (-1.1,0.25);
       \draw [->, very thick] (4.2,0.25)  to [out=90,in= 90] (-1.7,0.25);
	 \tikzstyle{every node}=[draw, circle, fill = white, draw = none, minimum size=1pt,
                            inner sep=0pt]      
       \draw (0.5,0.75) node {$a$};
       \draw (0.5,0.25) node {$b$};
       \draw (0.5, -0.25) node {$c$}; 
       \draw (2, 0.75) node {$d$};
       \draw (2, 0.25) node {$e$ };
       \draw (2, -0.25) node {$f$};
     \draw (-1.1, 0) node{$1$};
     \draw (1.25, 0) node{$2$};
     \draw(3.6, 0) node{$3$};
      \tikzstyle{every node}=[draw, circle, fill = white, draw = none, minimum size=1pt,
                            inner sep=0pt]    
       \draw (0.75, 1.17) node{$a_{q_1}$};
       \draw (1.5, 1.6) node{$a_{q_2}$};
       \draw (2.25, 1.9) node{$a_{q_3}$};

\end{tikzpicture}
\end{center}
We note that the new relations in the preprojective algebra $\Pi$ are quadratic relations. 

\end{example} 

Before stating the next theorem, we need the following lemma, which was established in the proof of \cite[Lemma 3.8]{AIR15}.

\begin{lemma}[{\cite{AIR15}}]
\label{gen}
Let $\Pi$ be a higher $(n+1)$-preprojective algebra. Let $P_{\bullet}$ be the minimal bimodule resolution of $\Pi$. Then each term $P_i$ is generated in degree $0$ or in degree $1$. Moreover, $P_0$ is generated in degree $0$ and $P_{n+1}$ is generated in degree $1$. 
\end{lemma}

\begin{proof}
By Remark \ref{rem:locally_finite}, the minimal bimodule resolution $P_{\bullet}$ of $\Pi$ has the self-duality property (\ref{self_dual}). Since it is minimal and $\Pi$ is positively graded by Definition \ref{def_bim_cy}, each $P_i$ is generated in non-negative degrees. Consider the isomorphism 
\[
	P_\bullet \cong P_\bullet^\vee[n+1](-1).
\]
If $P_i$ has a generator in degree $a$, then $P_{n+1-i}$ has a generator in degree $1-a$. Therefore, $1-a\geq 0$, which implies that $a = 0$ or $a=1$. Thus, each $P_i$ is generated in degree $0$ or $1$. Moreover, $\Pi$ is generated in degree $0$ as a bimodule over itself, so $P_0$ is generated in degree $0$, and thus $P_{n+1}$ is generated in degree $1$.
\end{proof}

\begin{theorem}
\label{thm:der_quotient_preproj}
Let $\Lambda = T_S V/\langle M\rangle$ be a basic $n$-representation-infinite algebra and $\Pi$ be the preprojective algebra over $\Lambda$. If $\Pi$ is Koszul, then $\Lambda$ is Koszul. As a partial converse, if $\Lambda= T_S V/\langle M\rangle$ is a basic finite-dimensional Koszul algebra of global dimension $n$, then 
\[
	\omega := \sum_{q\in \mathcal B(K_n)}\sum_{i=0}^n (-1)^{i\cdot n}\sigma_{n+1}^{i}(qa_q)
\]
is a superpotential of order $n+1$ in $k\overline Q$, where $\sigma_{n+1}$ is defined together with its action in (\ref{sigma_n}). Moreover, if, in addition, $n\leq 2$ or $\Lambda$ is $n$-representation-infinite, then there is an isomorphism of algebras
\[
	\Pi:=T_{\Lambda} \Ext^n_{\Lambda} (D\Lambda, \Lambda) \cong D(\omega, (n+1)-2).
\]
\end{theorem}

\begin{proof}
Suppose that $\Pi$ is Koszul. Consider its Koszul resolution 
\[
	0\to\Pi\otimes_S \tilde K_{n+1}\otimes_S\Pi \to\Pi\otimes_S\tilde K_n\otimes_S\Pi \cdots\to \Pi\otimes_S\Pi\to\Pi\to 0.
\]
Taking the degree $0$ part of this resolution with respect to the preprojective grading, defined in Remark \ref{grading_type}, we get a complex
\[
	0\to\Lambda\otimes_S K_n\otimes_S\Lambda\to\cdots\to \Lambda\otimes_S\Lambda\to \Lambda\to 0,
\]
where $K_{\ell} := (\tilde K_{\ell})_0$. Since $\tilde K_{n+1}$ is generated in degree $1$ by Lemma \ref{gen}, we have $K_{n+1} = 0$.  The Koszul resolution of $\Pi$ is exact in each degree separately, and thus, in particular in degree $0$. Therefore, the latter complex is exact and is the Koszul resolution of $\Lambda$. \\

Now, if $\Lambda$ is Koszul, then $\omega$ is a superpotential of order $n+1$. In fact, the choice of the arrows $a_q$ gives that $\omega$ commutes with the action of $S$. Moreover, it is super-cyclically symmetric by Lemma \ref{super_cycl}. \\

By Corollary \ref{construction}, and since $K_n = \bigcap_{\mu} (V^{\otimes\mu}\otimes M \otimes V^{\otimes n-\mu-2})$, we have that $ W_2 \subset  \overline M$, where $W_2$ is the $S$-bimodule generated by the elements $\delta_p\omega$, for $p\in \mathcal B(V^{\otimes n-1})$, introduced in Definition \ref{def:der_quotient}. Moreover, it is easy to see that the new relations $\tilde M$ in $\Pi$ satisfy $\tilde M\subset  W_2$. We thus need to show that $ M\subset W_2$ in order to conclude that 
\[
	\Pi\cong k\overline Q/\langle \overline M\rangle = k\overline Q/ \langle W_2\rangle= D(\omega, (n+1)-2).
\]
If $\Lambda$ is hereditary, this is trivial as $\Lambda$ has no relation. If $\mathsf {gl.dim}\, \Lambda = 2$, then 
\[
	\omega = \sum_{q\in \mathcal B(M)}\sum_{i=0}^2 \sigma_{3}^{i}(qa_q)
\]  
since $ K_2 =  M $, so $ M  \subset  W_2$ clearly. Now assume that $\Lambda$ is $n$-representation-infinite, where $n\geq 3$. By \cite[Lemma 5.2b)]{Iya11} and \cite[Theorem 3.4]{HIO14}, applying $\Hom_{\Lambda^e}(-, \Lambda^e)$ to the complex (\ref{eq:koszul_complex}) and using the isomorphisms $g_{\ell}$ yield a complex 
\[
	0\to \bigoplus_{\substack{e_i\in \mathcal B(K_0)}}\Lambda e_i\otimesk e_i\Lambda\xrightarrow{\alpha_1} \bigoplus_{\substack{p: i\to j\\ p\in \mathcal B(K_1)}}\Lambda e_i\otimesk e_j\Lambda\xrightarrow{ \alpha_2}\cdots\xrightarrow{ \alpha_n}\bigoplus_{\substack{p: i\to j\\ p\in \mathcal B(K_{n})}}\Lambda e_i\otimesk e_j\Lambda\to 0
\] 
which is exact in degrees $1, 2, \ldots, n-1$. Here $\alpha_{\ell} := g_{\ell}\circ \tilde f_{\ell}\circ g_{\ell-1}^{-1}$. For $1\leq \ell \leq n$, we define  
\[
	\delta(K_{\ell}):= \{p\in K_{\ell-1}\,\, |\,\, p = \delta^{\mathcal L}(q)\text{ or } p = \delta^{\mathcal R} (q)\text{ for some } q\in K_{\ell}\}.
\]
We show that $K_{\ell}= \delta(K_{\ell+1})$ for all $2\leq \ell \leq n-1$. In particular, we deduce $ M =  K_2 =  \delta^{n-2} (K_n) \subset  W_2 $, which completes the proof. We have that $\Image \alpha_{\ell}$ is generated by
\[
	 \left\{ \sum_{\substack{q\in \mathcal B(K_{\ell})}} (e_i\otimes e_j \delta^{\mathcal R}_p (q))_{q} + (-1)^\ell\sum_{\substack{q\in \mathcal B(K_\ell)}}  (\delta^{\mathcal L}_p (q) e_i\otimes e_j)_{q}\quad | \quad p\in \mathcal B(K_{\ell-1}) \right\}.
\]
Suppose by contradiction that there exists $\tilde q\in \mathcal B(K_{\ell})$ such that $\tilde q\not\in \delta (K_{\ell+1})$. Then $\tilde q\not = \delta^{\mathcal L} (q)$ and $\tilde q\not = \delta^{\mathcal R} (q)$ for any $q\in K_{\ell+1}$. Thus, $\alpha_{\ell+1}((e_i\otimes e_j)_{\tilde q}) = 0$, similarly to (\ref{eq:diff_f}). Therefore, $(e_i\otimes e_j)_{\tilde q}\in \ker \alpha_{\ell+1}$, but $(e_i\otimes e_j)_{\tilde q}\not\in \Image \alpha_\ell$, which contradicts the exactness of the complex in degree $\ell$. 
\end{proof}

\begin{remark}
This theorem generalizes the result \cite[Theorem 6.10]{Kel11}, where the author proves that a $3$-preprojective algebra is a Jacobian algebra given by a potential. Moreover, this also generalizes partially the main theorem in \cite[Theorem 3.4.2]{EE07}, which states that $2$-preprojective algebras over representation-infinite algebras are Koszul, and thus are described by derivation-quotient algebras. In \cite[Theorem 4.9]{GI19}, the authors showed that $\Pi(\Lambda)$ is Koszul if $\Lambda$ is an $n$-representation-infinite Koszul algebra.
\end{remark}


\section{Tensor product of preprojective algebras}

In this section, we study the preprojective grading structure on the tensor product of two bimodule Calabi--Yau algebras. When considering the natural tensor product grading on two graded $k$-algebras $A^1$ and $A^2$, given by 
\[
	(A^1\otimesk A^2)_i := \bigoplus_{\ell+m = i} A^1_{\ell}\otimesk A^2_m,
\]
we get the following proposition.

\begin{proposition}
\label{tensor_product_natural_grading}
If $A^i$ is an $n_i$-Calabi--Yau algebra of Gorenstein parameter $a_i$, for $i = 1, 2$, then $A:=A^1\otimesk A^2$, along with the natural tensor product grading, is $(n_1+n_2)$-Calabi--Yau of Gorenstein parameter $(a_1 + a_2)$. 
\end{proposition}

\begin{proof}
Since $A^i$ is an $n_i$-Calabi--Yau algebra of Gorenstein parameter $a_i$, $i=1,2$, we have by definition that $A^i\in\mathsf{per}\, (A^i)^e$ for $i=1,2$, that is, $A^i$ is quasi-isomorphic to a bounded complex $Q^i_{\bullet}$ of finite projective $(A^i)^e$-modules. Since $k$ is a field, the total complex $\Tot(Q_{\bullet}^1\otimesk Q_{\bullet}^2)$ is a bounded complex of finite projective $(A^1)^e\otimes (A^2)^e$-modules, which is quasi-isomorphic to $A^1\otimesk A^2$. Thus, $A\in\mathsf{per}\, (A^1\otimes_k A^2)^e \subset  \DD(\mathsf{Mod}\, (A^1\otimesk A^2)^e)$.\\

Moreover, there exists a graded bimodule projective resolution $P^i_{\bullet}$ of $A^i$ and an isomorphism  
\[
	P^i_{\bullet}\cong \Hom_{(A^i)^e}(P^i_{\bullet}, (A^i)^e)[n_i](-a_i)\quad\text{in } \mathsf C^b(\mathsf{grproj}\,(A^i)^e).
\]
We need to show the existence of an analogous isomorphism for a projective bimodule resolution $P_{\bullet}$ of $A$ in ${C^b(\mathsf{grproj}\, A^e)}$.\\

Again, since $k$ is a field, the total complex $P_{\bullet}:=\Tot(P_{\bullet}^1\otimesk P_{\bullet}^2)$ is a bimodule projective resolution of $A$. We claim that $P_{\bullet}$ has the desired self-duality property. In fact, 
\begin{align*}
\Tot(P_{\bullet}^1\otimesk P_{\bullet}^2)&\cong \Tot\left(\Hom_{(A^1)^e}(P_{\bullet}^1,(A^1)^e)[n_1](-a_1)\otimesk\Hom_{(A^2)^e}(P_{\bullet}^2, (A^2)^e)[n_2](-a_2)\right)\\
&\cong \Tot\left(\Hom_{(A^1\otimesk A^2)^e}(P_{\bullet}^1\otimesk P_{\bullet}^2, (A^1\otimesk A^2)^e)\right)[n_1+n_2](-a_1-a_2)\\
& \cong \Hom_{(A^1\otimesk A^2)^e}\left(\Tot(P_{\bullet}^1\otimesk P_{\bullet}^2), (A^1\otimesk A^2)^e\right)[n_1+n_2](-a_1-a_2).
\end{align*}

To see that the second isomorphism holds, note that the morphism
\begin{align*}
	\Hom_{(A^1)^e}(P^1_i, (A^1)^e)\otimesk\Hom_{(A^2)^e}(P^2_j, (A^2)^e)&\to \Hom_{(A^1)^e\otimesk (A^2)^e}(P_i^1\otimesk P_j^2, (A^1)^e\otimesk (A^2)^e)\\
	 f\otimes g\qquad\qquad\qquad\qquad\,\,\,\,\,\,&\mapsto\qquad\qquad\,\, (p\otimes q\mapsto f(p)\otimes g(q)),
\end{align*}
is an isomorphism. In fact, if $P^1_i= (A^1)^e$ and $P^2_j = (A^2)^e$, then the map factors through the natural isomorphisms 
\begin{align*}
	\Hom_{(A^1)^e}(P^1_i, (A^1)^e)\otimesk\Hom_{(A^2)^e}(P_2^j, (A^2)^e)&\cong (A^1)^e\otimesk (A^2)^e\\
	 &\cong \Hom_{(A^1)^e\otimesk (A^2)^e}(P_i^1\otimesk P_j^2, (A^1)^e\otimesk (A^2)^e).
\end{align*}

The case where $P_i^1$ and $P_j^2$ are finitely generated projective modules then follows. These morphisms commute with the differentials. 
\end{proof}

The main interest of this section is to determine whether we can endow the tensor product of two bimodule Calabi--Yau algebras with the grading structure of a preprojective algebra. Throughout, we let $A^i$ be a bimodule $n_i$-Calabi--Yau algebra, such that $n_i\geq 1$, for $i=1,2$. Since the tensor product of two Calabi--Yau algebras is always Calabi--Yau, it remains to determine if there exists a grading giving it Gorenstein parameter $1$ in such a way that the degree $0$ part is finite-dimensional. By the previous proposition, this grading cannot be the standard tensor product grading. If $A^1$ or $A^2$ has Gorenstein parameter $1$, then we can endow the tensor product with a grading giving it Gorenstein parameter $1$ as follows.  

\begin{lemma}
\label{trivial_grading}
If $A^1$ or $A^2$ admits a grading giving it Gorenstein parameter $1$, then so does the tensor product $A^1\otimesk A^2$. 
\end{lemma}

\begin{proof}
Suppose that $A^1$ has Gorenstein parameter $1$. We put $A^2$ in degree $0$ and keep the grading on $A^1$. Then the Gorenstein parameter of $A^1\otimesk A^2$ is the same as the Gorenstein parameter of $A^1$, which is $1$. 
\end{proof}

The degree $0$ part of the grading that we put on $A^1\otimesk A^2$ in Lemma \ref{trivial_grading} is of the form $\Lambda_1\otimesk \Lambda_2$, where $\Lambda_1 = A^1_0$ and $\Lambda_2 = A^2$ is $n_2$-Calabi--Yau. Note that $\Lambda_2$ is then infinite-dimensional by Remark \ref{rem:locally_finite}. However, we want to put a grading in such a way that the degree $0$ part is finite-dimensional, as this is an essential property in the definition of higher representation-infinite algebras. We show that this is impossible if $A^1$ and $A^2$ are Koszul. More precisely, we prove that the only possible grading giving $A^1\otimesk A^2$ Gorenstein parameter $1$ is the one described in the proof of Lemma \ref{trivial_grading}, with the grading induced either from the first or the second factor. This implies in particular that $A^1$ or $A^2$ must also admit a grading giving it Gorenstein parameter $1$. 

\begin{question}
The algebra $\Lambda_1\otimesk \Lambda_2$, while being infinite-dimensional, has finite global dimension and is the degree $0$ part of a bimodule Calabi--Yau algebra of Gorenstein parameter $1$. This raises the question whether we can extend the study of $n$-representation-infinite algebras to algebras of infinite dimension, and generalize the notion of preprojective algebras to bimodule Calabi--Yau algebras of Gorenstein parameter $1$ such that the degree $0$ part is not necessarily finite-dimensional. 
\end{question}

In \cite{HIO14}, the authors study the tensor product of two higher representation-infinite algebras. 

\begin{theorem}[{\cite[Theorem 2.10 \& Lemma 2.11]{HIO14}}]
\label{HIO}
Let $\Lambda_i$ be an $n_i$-representation-infinite algebra, for $i =1, 2$. Then 
\begin{itemize}
\item $\Lambda_1\otimesk\Lambda_2$ is $(n_1+n_2)$-representation-infinite;
\item $\Se(\Lambda_1\otimesk\Lambda_2) \cong \Se(\Lambda_1)\otimesk\Se(\Lambda_2)$;
\item $\Se_{n_1+n_2}(\Lambda_1\otimesk\Lambda_2) \cong \Se_{n_1}(\Lambda_1)\otimesk\Se_{n_2}(\Lambda_2)$.
\end{itemize}
\end{theorem}

Combining these results together, we get the following.

\begin{corollary}
Let $\Lambda_i$ be as above and let $\Pi_i = \Pi(\Lambda_i)$, for $i =1, 2$, be the preprojective algebra associated to $\Lambda_i$. Then the Segre product 
\[
	\Pi_1*\Pi_2:= \bigoplus_{\ell\geq 0}\left( (\Pi_1)_{\ell}\otimesk(\Pi_2)_{\ell}\right)
\]
is the $(n_1+n_2+1)$-preprojective algebra over $\Lambda_1\otimesk\Lambda_2$.
\end{corollary}

\begin{proof}
We have that 
\[
	\Pi_1*\Pi_2 = \bigoplus_{\ell\geq 0}\left( (\Pi_1)_{\ell}\otimesk(\Pi_2)_{\ell}\right) \cong \bigoplus_{\ell\geq 0}\left( \Se_{n_1}^{-\ell}(\Lambda_1)\otimesk  \Se_{n_2}^{-\ell}(\Lambda_2)\right) \cong \bigoplus_{\ell\geq 0} \Se_{n_1+n_2}^{-\ell}\left(\Lambda_1\otimesk\Lambda_2\right).
\]
Since $\Lambda_1\otimesk\Lambda_2$ is $(n_1+n_2)$-representation-infinite, the result follows.
\end{proof}

\begin{remark}
\label{remark:geometry}
We showed so far in this section that the tensor product of two preprojective algebras along with the natural grading is not a higher preprojective algebra. In contrast, the Segre product is the preprojective algebra of a tensor product of higher representation-infinite algebras. This intuitively makes sense from a geometric point-of-view. Suppose that $Z$ is a smooth projective Fano variety and let $Y :=  \operatorname{Tot}\nolimits (\omega_Z)$ be the total space of the canonical line bundle of $Z$. Denote the bundle map by $\varphi: Y\to Z$. Then $Y$ is a Calabi--Yau variety. We can view $Y$ as a geometric analogue of a preprojective algebra as follows. Let $T\in \bd(\mathsf{Coh}\, Z)$ be a tilting object in the bounded derived category of coherent sheaves on $Z$ and $\Lambda = \End_{Z}(T)$, so that $\bd(\Lambda)$ is equivalent to $\bd(\mathsf{Coh}\, Z)$. Then, by \cite[Theorem 3.6]{BS10}, there is an equivalence 
\[
	\DD(\Pi(\Lambda))\xrightarrow{\sim} \bd(\mathsf{Coh}\,Y),
\]
sending $\Pi(\Lambda)$ to $\varphi^*(T)$. Now suppose that $Z := Z_1\times Z_2$ is given by a product of smooth projective Fano varieties and let $\pi_i: Z\to Z_i$ be the projections for $i=1,2$. If $T_i\in \bd(\mathsf{Coh}\, Z_i)$ are tilting objects, then $T := \pi_1^*T_1\otimes \pi_2^*T_2$ is a tilting object in $\bd(\mathsf{Coh}\, Z)$. We then get an equivalence 
\[
	\bd(\Lambda_1\otimesk \Lambda_2)\xrightarrow{\sim}\bd(\mathsf{Coh}\,Z_1\times Z_2),
\]
where $\Lambda_i = \End_Z(\pi_i^*(T_i))$, for $i=1,2$, which gives a geometric analogue to the first point in Theorem \ref{HIO}.
 Thus, there is an equivalence 
 \[
\DD(\Pi(\Lambda_1)*\Pi(\Lambda_2))\xrightarrow{\sim} \bd(\mathsf{Coh}\,\operatorname{Tot}\nolimits (\omega_{Z_1}\otimesk\omega_{Z_2})).
\] 
The tensor product gives on the other hand an equivalence 
 \[
 	\DD(\Pi(\Lambda_1)\otimesk\Pi(\Lambda_2))\xrightarrow{\sim} \bd(\mathsf{Coh}\, \operatorname{Tot}\nolimits (\omega_{Z_1})\times \operatorname{Tot}\nolimits (\omega_{Z_2}) ).
 \]    
 Since $\operatorname{Tot}\nolimits (\omega_{Z_1})\times \operatorname{Tot}\nolimits (\omega_{Z_2})\cong \operatorname{Tot}\nolimits(\omega_{Z_1}\oplus\omega_{Z_2})$, we see that it is naturally isomorphic to the total space of a rank $2$ vector bundle, and not to the total space of a line bundle. Hence we should not expect the tensor product of two preprojective algebras to have a natural structure of preprojective algebra. This is coherent with Proposition \ref{tensor_product_natural_grading}. This does not however tell us much about the existence of any "non-natural" grading structure that would allow us to express $\Pi = \Pi(\Lambda_1)\otimesk\Pi(\Lambda_2)$ as a higher preprojective algebra.  Such gradings, even if not coming from the natural grading structure, would mean that $\Pi(\Lambda_1)\otimesk\Pi(\Lambda_2)$ has some interesting properties. We would for example then obtain a tilting object in $\underline{\mathsf{CM}}^{\Z}(e\Pi e)$ for an idempotent $e$ satisfying some additional properties (\cite[Theorem 4.1]{AIR15}), as mentioned in the introduction. 
 \end{remark}


\subsection{Superpotential of a tensor product}

Suppose that $A^i = T_{S_i} V_i/\langle M_i \rangle$ is an $n_i$-Calabi--Yau Koszul algebra, where $n_i\geq1$, $S_i$ is a finite-dimensional semisimple $k$-algebra, $V_i$ is a $S_i$-bimodule and $M_i$ are relations, for $i=1,2$. Then, $A:=A^1\otimesk A^2$ is $(n_1+n_2)$-Calabi--Yau and Koszul. It is thus a derivation-quotient algebra of the form $D(\omega, n-2)$ for some superpotential $\omega$, by Theorem \ref{BSW}. We would like to describe this superpotential in terms of $\omega_1$ and $\omega_2$, the superpotentials in $T_{S_1} V_1$ and $T_{S_2} V_2$, respectively, such that $A^i \cong D(\omega_i, n_i-2)$, for $i=1,2$. This is needed for the main theorem of this section. Note that we sometimes refer to $\omega_i$ (resp. $\omega$) simply as the superpotentials of $A^i$ (resp. $A$).  \\

The algebra $A$ has a tensor algebra structure given as follows: 
\[
	A = T_S V/\langle M\rangle, 
\]
where $S:=S_1\otimesk S_2$, $V = (V_1\otimesk S_2\oplus S_1\otimesk V_2)$, and $M$ is the $S$-bimodule of relations in $A$, induced from the relations in $A^1$ and $A^2$ and new relations coming from the tensor product. The space $V$ has a natural $S$-bimodule structure. Indeed, if $a\otimesk b\in V$, then $s = s_1\otimesk s_2$ acts on the right as 
\[
	(a\otimesk b)\cdot (s_1\otimesk s_2) = (a\cdot s_1)\otimesk (b\cdot s_2),
\]  
and the action extends linearly. The left action is similar.  \\

Recall from Remark \ref{sp_gen} that a superpotential in $A$ is a $S$-bimodule generator of 
\[
	K_n:=\bigcap_{\mu} (V^{\otimes\mu}\otimes_S M \otimes_S V^{\otimes n-\mu-2}),
\]
where $n = n_1+n_2$. We first give the definition of the shuffle product. 

\begin{definition}[See, e.g., {\cite[6.5.11 \& 8.5.4]{Wei94}}]
The \emph{shuffle product} 
\[
	\shuffle: V_1^{\otimes_{S_1} n_1}\otimes_k V_2^{\otimes_{S_2} n_2}\to V^{\otimes_{S} n}
\]
is defined as
\begin{equation*}
\begin{split}
	v_{n_1+n_2}\otimes\cdots&\otimes v_{n_2+1}\shuffle v_{n_2}\otimes\cdots\otimes v_{1} := \\
	 &\sum_{\tau \text{ shuffle}} (-1)^{\tau} \tilde v_{\tau^{-1} (n_1+n_2)}\otimes \tilde v_{\tau^{-1} (n_1+n_2-1)}\otimes\cdots\otimes \tilde v_{\tau^{-1} (n_2+1)}\otimes \tilde v_{\tau^{-1} (n_2)}\otimes\cdots\otimes \tilde v_{\tau^{-1} (1)},
\end{split}
\end{equation*}
where the sum runs over all \emph{$(n_1, n_2)$-shuffles} $\tau$, that is, elements $\tau$ of the symmetric group $\mathfrak{S}_{n_1+n_2}$ which satisfy $\tau(1)<\tau(2)<\cdots<\tau(n_2)$ and $\tau(n_2+1)<\tau(n_2+2)<\cdots<\tau(n_1+n_2)$. Here $\tilde v_i\in V$ denotes $v_i \otimesk 1$ if $v_i\in V_1$ and $1\otimesk v_i$ if $v_i\in V_2$. 
\end{definition}

\begin{theorem}
\label{shuffle}
Let $\omega_i$ be a superpotential in $T_{S_i} V_i$ such that $A^i\cong D(\omega_i, n_i-2)$, for $i=1,2$. The element  $\omega := \omega_1\shuffle \omega_2$ is a superpotential in $T_S V$ and $A \cong D(\omega, n-2)$. 
\end{theorem}

\begin{proof}
We first show that $\omega$ is a $S$-bimodule generator of $K_n$. As mentioned in (\ref{tor_kl}), there are isomorphisms 
\[
	\Tor^{A^i}_{n_i}(S_i, S_i)\cong\bigcap_{\mu} (V_i^{\otimes\mu}\otimes_{S_i} M_i \otimes_{S_i} {V_i}^{\otimes n_i-\mu-2})=:K_{n_i}^i. 
\]
for $i=1,2$, coming from the Koszul property. We may thus identify the superpotential $\omega_i\in K_{n_i}^i$ with a $S_i$-bimodule generator $\tilde\omega_i$ of $\Tor_{n_i}^{A^i}(S_i, S_i)$, for $i =1, 2$. \\

The shuffle product induces a chain homotopy equivalence, see for example \cite[Proposition 8.6.13]{Wei94}: 
\[
	\Tot(\mathbb B(A^1, S_1)\otimesk\mathbb B(A^2, S_2))\xrightarrow{\sim}\mathbb B(A^1\otimesk A^2, S_1\otimesk S_2),
\]
where Tot denotes the total complex of the tensor product and $\mathbb B$ is the bar resolution, defined in (\ref{bar_res}). We then obtain a $S$-bimodule isomorphism 
\[
	\shuffle: \Tor^{A^1}_{n_1}(S_1, S_1)\otimesk \Tor^{A^2}_{n_2}(S_2, S_2)\xrightarrow{\sim} \Tor^A_{n_1+n_2}(S_1\otimesk S_2, S_1\otimesk S_2).
\]
Thus, the image $\tilde\omega_1\shuffle\tilde\omega_2$ of $\tilde\omega_1\otimes\tilde\omega_2$ via the shuffle map is a $S$-bimodule generator of 
\[
	\Tor_{n_1+n_2}^A(S_1\otimesk S_2, S_1\otimesk S_2)\cong K_n.
\]
Combining with the fact that $\omega$ is a superpotential, which we prove next, this implies that the complex $W_{\bullet}$, defined in (\ref{eq:w}), is the Koszul resolution of $A$. Thus, $W_2 = M$, so  
\[
	A\cong T_S V/\langle W_2\rangle= D(\omega, n-2).
\] 

We now show that $\omega$ commutes with the action of $S$. If $s = s_1\otimesk s_2\in S$, then 
\[
	(\omega_1\shuffle \omega_2)\cdot s = (\omega_1\cdot s_1)\shuffle (\omega_2\cdot s_2) = (s_1\cdot \omega_1)\shuffle (s_2\cdot\omega_2) = s\cdot(\omega_1\shuffle\omega_2). 
\]
The second equality comes from the fact that $\omega_i$ commutes with the action of $S_i$, for $i=1,2$. The first and third equalities are true for any shuffle of elements in $V_1^{\otimes n_1}$ and $V_2^{\otimes n_2}$. In fact, 
\begin{equation*}
\begin{split}
	(v_{n_1+n_2}\otimes\cdots\otimes& v_{n_2+1}\shuffle v_{n_2}\otimes\cdots\otimes v_{1})\cdot (s_1\otimesk s_2) =\\
	 \sum_{\substack{\tau \text{ shuffle} \\ \tau(n_2+1) = 1}} (-1)^{\tau} &\tilde v_{\tau^{-1} (n_1+n_2)}\otimes \tilde v_{\tau^{-1} (n_1+n_2-1)}\otimes\cdots\otimes \tilde v_{\tau^{-1} (n_2+1)} \otimes \tilde v_{\tau^{-1} (n_2)}\otimes\cdots\otimes(1\otimesk v_{1}\cdot s_2) \\
	 &\otimes(v_{\tau^{-1}(\tau(1)-1)}\otimesk 1)\otimes\cdots\otimes (v_{n_2+1}\cdot s_1\otimesk 1) \\
	+ \sum_{\substack{\tau \text{ shuffle} \\ \tau(1) = 1}} (-1)^{\tau} &\tilde v_{\tau^{-1} (n_1+n_2)}\otimes \tilde v_{\tau^{-1} (n_1+n_2-1)}\otimes\cdots\otimes \tilde v_{\tau^{-1} (n_2+1)} \otimes \tilde v_{\tau^{-1} (n_2)}\otimes\cdots\otimes(v_{n_{2}+1}\cdot s_1\otimesk 1) \\
	&\otimes(1\otimesk v_{\tau^{-1}(\tau(n_2+1)-1)} )\otimes\cdots\otimes (1\otimesk v_{1}\cdot s_2),
\end{split}
\end{equation*}
which is equal to $(v_{n_1+n_2}\otimes\cdots\otimes v_{n_2+1}\cdot s_1)\shuffle (v_{n_2}\otimes\cdots\otimes v_{1}\cdot s_2)$. The third equality is proven similarly.\\

Finally, we prove that $\omega$ is super-cyclically symmetric. By Lemma \ref{super_cycl}, the superpotentials $\omega_1$ and $\omega_2$ are given as a linear combination of elements of the form 
\begin{equation*}
\begin{split}
	\sum_{i=0}^{n_1-1}(-1)^{i(n_1-1)} \sigma_{n_1}^i(v_{n_1+n_2}\otimes\cdots\otimes v_{n_2+1})\quad\text{and}\quad
	\sum_{j=0}^{n_2-1}(-1)^{j(n_2-1)} \sigma_{n_2}^j(v_{n_2}\otimes\cdots\otimes v_{1})
\end{split}
\end{equation*}
respectively. Thus, $\omega_1\shuffle\omega_2$ is given as a linear combination of elements of the form 
\[
	\overline{\omega_1\shuffle \omega_2}:= \sum_{i=0}^{n_1-1}\sum_{j=0}^{n_2-1}(-1)^{i(n_1-1)+j(n_2-1)}(\sigma_{n_1}^i(v_{n_1+n_2}\otimes\cdots\otimes v_{n_2+1})\shuffle  \sigma_{n_2}^j(v_{n_2}\otimes\cdots\otimes v_{1})).
\]
It suffices to show that these are super-cyclically symmetric. In the following we shall use the notation $\shuffle_{\tau(i) = j}$ to denote the sum over shuffles satisfying $\tau(i) = j$ and we let $a_{i,j} = i(n_1-1)+j(n_2-1)$. By rearranging the order of the terms, we have that $ \overline{\omega_1\shuffle \omega_2}$ is equal to
\begin{align*}
\begin{split}
 &\,\,\,\, \sum_{i=0}^{n_1-1}\sum_{j=0}^{n_2-1}(-1)^{a_{i,j}}((-1)^{n_1-1}\sigma_{n_1}^{i-1}(v_{n_1+n_2}\otimes\cdots\otimes v_{n_2+1})\shuffle_{\tau(n_1+n_2 -i +1) = n_1+n_2} \sigma_{n_2}^j(v_{n_2}\otimes\cdots\otimes v_{1})\\
 &\qquad\qquad\qquad\qquad+(-1)^{n_2-1} \sigma_{n_1}^{i}(v_{n_1+n_2}\otimes\cdots\otimes v_{n_2+1})\shuffle_{\tau(n_2-j+1) = n_1+n_2} \sigma_{n_2}^{j-1}(v_{n_2}\otimes\cdots\otimes v_{1}))
\end{split}
\\
\begin{split}
 &= \sigma_n^{-1}\left(\sum_{i=0}^{n_1-1}\sum_{j=0}^{n_2-1}(-1)^{a_{i,j}}((-1)^{n-1}\sigma_{n_1}^{i}(v_{n_1+n_2}\otimes\cdots\otimes v_{n_2+1})\shuffle \sigma_{n_2}^j(v_{n_2}\otimes\cdots\otimes v_{1}))\right) 
  \end{split}
 \\
 & = (-1)^{n-1}\sigma_n^{-1}(\overline{\omega_1\shuffle \omega_2}).
\end{align*} 
\end{proof}


\subsection{Preprojective structure on the tensor product of preprojective algebras}

The goal of this subsection is to prove that the tensor product of two basic Koszul Calabi--Yau algebras cannot have a grading giving it the structure of a preprojective algebra.  We begin by proving some general statements. \\

Let $A^i = kQ^i/\langle M_i \rangle$ be a basic algebra, where $Q^i = (Q^i_0, Q^i_1)$ is a quiver, for $i=1,2$. We recall the description of the quiver of the tensor product $A:= A^1\otimesk A^2$. Following \cite[Section 3]{GMV98}, the quiver $Q = (Q_0, Q_1)$ of $A$ is described as follows: 
\[
	Q_0 = Q^1_0\times Q^2_0
\]
and 
\[
	Q_1 = (Q^1_1\times Q^2_0)\cup(Q^1_0\times Q^2_1).
\]
Thus, for every arrow $a^1: e_i^1\to e_j^1$ in $Q^1$, there are arrows
\[
	(a^1, e_s^2): (e_i^1,e_s^2)\to(e_j^1, e_s^2),
\]
for every $e_s^2\in Q^2$. Similarly, we get the arrows of type $(e_s^1, a^2)$. Moreover, for every linear combination of paths $\sum\mu_i p_i^1\in Q^1$, where $p_i^1 = a_{i_{\ell}}^1\cdots a_{i_1}^1$, $\mu_i\in k$, and every vertex $e_s^2\in Q^2$, there exists a linear combination of paths
\[
	\sum \mu_i (a_{i_{\ell}}^1, e_s^2)\cdots (a_{i_1}^1, e_s^2).
\]
The paths in the second component are obtained in a similar way. \\

For every arrow $a^1: e_i^1\to e_j^1\in Q^1$ and $a^2: e_s^2\to e_t^2\in Q^2$, let $\overline M$ be the ideal generated by the relations
\[
	(a^1, e_t^2)(e_i^1, a^2)-(e_j^1, a^2)(a^1, e^2_s).
\]
Then $A\cong kQ/\langle \tilde M_1, \tilde M_2, \overline M\rangle$, where 
\[
	\tilde M_1 = \{(f^1_i, e'')\,\, | \,\, f^1_i\in M_1, e''\in Q^2_0\}
\]
and 
\[
	\tilde M_2 = \{(e', f^2_j)\,\, | \,\, f^2_j\in M_2, e'\in Q^1_0\}. 
\]
When $A^1$ and $A^2$ are Koszul, then the Koszul grading on $A$ is the standard one given by the tensor product. \\

Now let $A = T_S V/\langle M\rangle $ be a bimodule $n$-Calabi--Yau Koszul algebra of Gorenstein parameter $1$. Let $P_{\bullet}$ denote the minimal bimodule resolution of $A$. Recall that the terms in $P_{\bullet}$ are given by 
\[
	P_{\ell} = A\otimes_S K_{\ell}\otimes_S A,
\]
where $K_{\ell}\subset V^{\otimes_S \ell}$ is the usual term in the Koszul resolution given in Definition \ref{koszul_complex}. Let $1\otimes v_{\ell}\otimes \cdots \otimes v_1\otimes 1\in P_{\ell}$, where we now view $P_{\ell}$ as a graded $A$-bimodule with respect to the preprojective grading, defined in Remark \ref{grading_type}. There is a corresponding element $v_{\ell}\otimes \cdots \otimes v_1$ in the ring $T_S V$, also endowed with a preprojective grading. In many instances, we will use the fact that these two elements have the same preprojective degree, denoted by $\deg_{P_{\ell}}$ and $\deg_{T_S V}$, respectively. This is not necessarily clear a priori, so we first check that this is indeed the case.      

\begin{lemma}
\label{key}
Let $A = T_S V/\langle M\rangle$ be a Koszul $n$-Calabi--Yau algebra of Gorenstein parameter $1$. Then
\[
	\deg_{P_{\ell}}(1\otimes v_{\ell}\otimes\cdots\otimes v_1\otimes 1) = \deg_{T_S V}(v_{\ell}\otimes\cdots\otimes v_1), 
\]
for any basis $v_{\ell}\otimes\cdots\otimes v_1\in K_{\ell}$. 
\end{lemma}

\begin{proof}
We proceed by induction on ${\ell}$ and use the fact that the differentials $d$ in the resolution are homogeneous with respect to the preprojective grading. For $\ell=1$, we have 
\[
	d(1\otimes v_1\otimes 1) = v_1\otimes 1 - 1\otimes v_1\in P_0=A\otimes_SA.
\]
Hence, 
\[
	\deg_{P_1}(1\otimes v_1\otimes 1) = \deg_{P_0}(v_1\otimes 1 - 1\otimes v_1).
\]
The bimodule $A\otimes_SA$ is generated in degree $0$ by $1\otimes 1$, so 
\[
	\deg_{P_0}(a\cdot (1\otimes1)\cdot b) = \deg_{T_S V} (a) + \deg_{T_S V} (b),
\]
for any $a,b\in T_S V$. 
 We conclude that 
\[
	\deg_{P_0}(v_1\otimes 1 - 1 \otimes v_1) = \deg_{P_0}(v_1\cdot (1\otimes 1) - (1\otimes 1)\cdot v_1) = \deg_{T_S V}(v_1). 
\]  
Applying induction and using 
\[
	d(1\otimes (v_{\ell}\otimes\cdots\otimes v_{1})\otimes 1) = v_{\ell}\otimes (v_{\ell-1}\otimes\cdots\otimes v_{1})\otimes 1 + (-1)^{\ell} 1\otimes (v_{\ell}\otimes\cdots\otimes v_{2})\otimes v_{1},
\]
we obtain 
\begin{align*}
	\deg_{P_{\ell}}(1\otimes v_{\ell}\otimes\cdots\otimes &v_1\otimes 1)\\
	 &= \deg_{P_{\ell-1}}(v_{\ell}\cdot 1\otimes (v_{\ell-1}\otimes\cdots\otimes v_{1})\otimes 1 + (-1)^{\ell} 1\otimes (v_{\ell}\otimes\cdots\otimes v_{2})\otimes 1\cdot v_{1})\\
	&= \deg_{T_S V}(v_{\ell}) + \deg_{T_S V} (v_{\ell-1}\otimes \cdots\otimes v_1)\\
	& = \deg_{T_S V}(v_{\ell}\otimes \cdots \otimes v_1).
\end{align*}
\end{proof}

If $A$ is basic, then the superpotential $\omega$ is given as a linear combinations of closed paths of length $n$, since it commutes with the $S$-action. We refer to those paths as \emph{(closed) summands} of $\omega$. Before stating the main theorem, we need the following lemma. 

\begin{lemma}
\label{lem:arrow_cont}
Let $A= T_S V/\langle M\rangle$ be a basic $n$-Calabi--Yau Koszul algebra with quiver $Q$ and superpotential $\omega$. Every arrow $a\in Q$ is contained in a closed cycle which is a summand of $\omega$.  
\end{lemma}

\begin{proof}
By Proposition \ref{prop:cy_duality}, there is an isomorphism 
\[
	\Ext^1_A(X,Y)\cong D\Ext^{n-1}_A(Y,X), 
\] 
for any finite-dimensional left $A$-modules $X, Y$. Combining with the isomorphism (\ref{tor_kl}), we obtain 
\[
	\dim_k e_jK_1e_i=\dim_k\Ext_A^1(S_i, S_j) = \dim_k\Ext^{n-1}_A(S_j, S_i) = \dim_k e_i K_{n-1} e_j, 
\]  
for any $i, j\in S$, where $S_i$ denotes the simple module at vertex $i$. Now, using Remark \ref{sp_gen} and the fact that $K_1 = V$, we conclude that the number of arrows $i\to j$ is the same as the number of (non-zero) elements of the form $e_i\delta_a \omega e_j$ from $j\to i$, with $a\in V$. Hence, every arrow $a:i\to j$ is in a closed cycle in $a e_i\delta_a\omega e_j$, which is a summand $\omega$.   
\end{proof}

Let $A^i = T_{S_i} V_i/\langle M_i \rangle$ be a basic bimodule $n_i$-Calabi--Yau Koszul algebra, for $i =1, 2$, where $n_1, n_2\geq 1$. Assume in addition that the quiver $Q^i$ of $A^i$ is connected. Then 
\[
	A: = A^1\otimesk A^2 \cong T_S V/\langle M \rangle,
\]
is an $n$-Calabi--Yau Koszul algebra, where $n = n_1 + n_2$, $S:= S_1\otimesk S_2$, $V = V_1\otimesk S_2\oplus S_1\otimesk V_2$ and $M$ are induced relations. Let $\omega_i$ be a superpotential of $A^i$, $i=1,2$, and $\omega: = \omega_1\shuffle\omega_2$ be a superpotential of $A$, according to Theorem \ref{shuffle}. \\

Recall from Remark \ref{rem:locally_finite} that an important feature of preprojective algebras is that their degree $0$ part is finite-dimensional. This is what breaks when trying to put a grading structure of preprojective algebra on $A^1\otimesk A^2$. 

\begin{theorem}
\label{no_tensor}

Let $A^1$ and $A^2$ be as above and let $A: = A^1\otimesk A^2$. If $A$ admits a grading such that the minimal graded $A$-bimodule resolution $P_{\bullet}$ of $A$ satisfies 
\[
	P_{\bullet}\cong P_{\bullet}^{\vee}[n](-1)\quad\text{in } \mathsf C^b(\mathsf{grproj}\,A^e),
\] 
then the degree $0$ part is of the form $\Lambda_1\otimesk \Lambda_2$, where $\Lambda_i = A^i_0$ and $\Lambda_j= A^j$ is $n_j$-Calabi--Yau, for some $i\not=j\in\{1,2\}$. Moreover, $A^i$ admits a grading giving it Gorenstein parameter $1$.   \\

In particular, $A_0$ must be infinite-dimensional. Hence, there does not exist a grading on $A$ giving it the structure of a preprojective algebra.

\end{theorem}

\begin{proof}
Let $Q^i$ be the quiver of $A^i$, for $i =1, 2$ and $Q$ be the quiver of $A$. For a path $p: i\to j$, we define $t(p) := i$ and $h(p):=j$.\\

By Lemma \ref{gen}, the $S$-bimodule generator $\omega\in K_{n_1+n_2}$ in $P_{\bullet}$, which is the superpotential of $A$, is homogeneous of degree $1$. Moreover, by Lemma \ref{key}, the degree of the elements in $K_{\ell}$ is the same as the degree of the associated paths in the ring $T_S V$. \\

The superpotential $\omega$ of $A$ is of the form $\omega = \omega_1\shuffle\omega_2$. By definition, all summands $\mathfrak p$ of $\omega_2\otimesk\omega_1$ are summands of $\omega$. They are closed paths $(e, e')\to (e,e')$ of length $(n_1+n_2)$, for some vertex $(e, e')$ in $Q_0$, and consist of the concatenation of two closed summands $\mathfrak p_1, \mathfrak p_2: (e, e')\to (e, e')$ of length $n_1$ and $n_2$ in $(\omega_1, e')$ and $(e,\omega_2)$, respectively. This is explained by the fact that both $\omega_1$ and $\omega_2$ commute with the action of $S$.\\

Consider one of these summands $\mathfrak p = \mathfrak p_2\mathfrak p_1 =(e, q)(p, e')$, where $p$ is a closed summand of $\omega_1$, and $q$ is a closed summand of $\omega_2$. Since $\omega$ is homogeneous of degree $1$, the path $(e, q)(p, e')$ is in degree $1$. By additivity of the degrees, either $(e, q)$ or $(p, e')$ is in degree $1$, say $(p, e')$ without loss of generality. Then $(e, q)$ is in degree $0$. Denote by $(a, e')$ the arrow in $(p, e')$ which is in degree $1$.\\

We divide the proof into the following steps. 
\begin{enumerate}[1)]
\item  We show that $\deg_{T_S V}(p,\varepsilon') = 1$ for any vertex $\varepsilon'$ appearing in the path $q$ and that $\deg_{T_S V}(\varepsilon,q) = 0$ for any vertex $\varepsilon$ appearing in the path $p$.
\item Using Lemma \ref{lem:arrow_cont} and the connectivity of $Q$, we proceed to demonstrate that \\$\deg_{T_S V}(p, \varepsilon') = 1$ for any vertex $\varepsilon'\in Q^2_0$ and $\deg_{T_S V}(\varepsilon, q) = 0$ for any vertex $\varepsilon\in Q^1_0$.
\item Let $\varrho$ be another summand of $\omega_1$. We show that $\deg_{T_S V}(\varrho, \varepsilon') =1$ for any vertex $\varepsilon'\in Q^2_0$. Likewise, we obtain that $\deg_{T_S V}(\varepsilon, \varrho') =1$ for any vertex $\varepsilon\in Q^1_0$ and closed summand $\varrho'$ of $\omega_2$. This implies that for any vertex $\varepsilon\in Q^1_0$, $\deg_{T_S V}(\varepsilon, \omega_2) = 0$ and for any $\varepsilon'\in Q^2_0$, $\deg_{T_S V}(\omega_1, \varepsilon') = 1$. 
\item We finally show that $\deg_{T_S V}(a, \varepsilon') = 1$ for any $\varepsilon'\in Q^2_0$.   

\end{enumerate}

Combining these steps, we conclude that for any arrow $\alpha\in Q^1$ and vertex $\varepsilon'\in Q^2_0$,
\[
	\deg_{T_S V}(\alpha, \varepsilon') = \deg_{T_{S_1} V_1}(\alpha),
\] 
so their degree is induced from the degree of the corresponding arrow in $A^1$. In particular, $A^1$ must have Gorenstein parameter $1$. Similarly, for any arrow $\beta\in Q^2$ and vertex $\varepsilon\in Q^1$,  
\[
	\deg_{T_S V}(\varepsilon, \beta) = 0.
\]
Thus, the degree $0$ part of $A$ must be of the form $\Lambda_1\otimesk \Lambda_2$, where $\Lambda_1 = A^1_0$ and $\Lambda_2 = A^2$ is $n_2$-Calabi--Yau.\\

\noindent \underline{\emph{Step 1}}: Let $\sigma_{n_1}\in\mathfrak S_{n_1}$ be as in (\ref{sigma_n}). The path $(\sigma_{n_1}^r(p), e')$ is in degree $1$ for all $0\leq r < n_1$, since it contains $(a, e')$. Thus, for any vertex $\varepsilon\in Q^1_0$ such that $\varepsilon = h(\sigma_{n_1}^r(p))$ for some $r$, the path $(\varepsilon, q)(\sigma_{n_1}^r(p), e')$ is in degree $1$. Therefore $(\varepsilon, q)$ is in degree $0$ for any vertex $\varepsilon$ in the path $p$. By the same reasoning, for every vertex $\varepsilon'$ in $q$, the path $(p, \varepsilon')$ is in degree $1$. \\

\noindent \underline{\emph{Step 2}}: Now consider a vertex ${\ell}\in Q^2_0$ which is not in the path $q$. We claim that the path $(p, \ell)$ is also in degree $1$. If $i\to j$ is an arrow, we use the symbol $i\sharp j$ to denote the underlying edge. Consider a non-oriented path 
\[
	(e, f_1) \sharp (e, f_2) \sharp \cdots \sharp (e, \ell),
\]
where $f_1\in Q^2_0$ is a vertex in the path $q$. Such a path exists since the quiver is connected. Recall that the path $(p, f_1): (e, f_1)\to (e, f_1)$ is in degree $1$ by the previous step. By Lemma \ref{lem:arrow_cont}, the arrow between $(e, f_1)$ and $(e, f_2)$ is part of a closed summand of $(e,\omega_2)$, call it $(e, q')$. Then $(e, \sigma_{n_2}^s(q'))(p, f_1)$ is a closed summand of $\omega$ in degree $1$, where $s$ is chosen so that $t(\sigma_{n_2}^s(q')) = f_1$. Hence $(e, \sigma_{n_2}^s(q'))$ has degree $0$, and so does $(e, \sigma_{n_2}^{r}(q'))$ for any $0\leq r<n_2$, since it contains the same arrows. Thus, the cycle $(p, f_2)$ must be in degree $1$, because $(e, \sigma_{n_2}^{s'}(q'))(p, f_2)$, where $s'$ is chosen such that $t(\sigma_{n_2}^{s'}(q')) = f_2$, is a summand of $\omega$ in degree $1$. Continuing in this fashion, we see that the path $(p, \ell)$ is in degree $1$. Thus, any path of the form $(p, \varepsilon')$, where $\varepsilon'\in Q^2_0$, is in degree $1$. By the same reasoning, any path of the form $(\varepsilon, q)$, where $\varepsilon\in Q$ is in degree $0$. \\

\noindent \underline{\emph{Step 3}}: Now consider another closed summand $\varrho$ of $\omega_1$ ending at a vertex $\varepsilon$. The closed summand $(\varepsilon, q)(\varrho, e')$ of $\omega$ is in degree $1$. By the previous step, $(\varepsilon, q)$ is in degree $0$, so $(\varrho, e')$ must be in degree $1$. This implies that $(\varrho, \varepsilon')$ is in degree $1$ for any vertex $\varepsilon'\in Q^2_0$. Similarly, we get that any $(\varepsilon, \varrho')$ is in degree $0$ for any summand $\varrho'$ of $\omega_2$ and vertex $\varepsilon\in Q^1_0$. \\

\noindent \underline{\emph{Step 4}}: It remains to show that if $(a,e)$ is the arrow in $(p, e)$ in degree $1$, then $(a, \varepsilon')$ is in degree $1$ for any $\varepsilon'\in Q^2_0$. Say $a:i\to j$ in $Q^1$. Consider as before a non-oriented path 
\[
	(i, e) \sharp (i, f_1) \sharp \cdots \sharp (i, \varepsilon')
\]	
and the parallel non-oriented path
\[
	(j, e) \sharp (j, f_1) \sharp \cdots \sharp (j, \varepsilon').
\]
Then, consider an arrow $b_1\in A^2$ between $e$ and $f_1$, say, without loss of generality, that $b_1: e\to f_1$. Then by the commutativity relations on the tensor product, we have
\[
	(a, f_1)(i, b_1) = (j, b_1)(a, e).
\]
Because $b_1\in A^2$, the arrows $(i, b_1)$ and $(j, b_1)$ are in degree $0$. In fact, we proved in step 3 that $\deg_{T_S V}(\varepsilon, \omega_2) = 0$ for any vertex $\varepsilon\in Q^1_0$. Moreover, $(a,e)$ is in degree $1$ by assumption. So $(a, f_1)$ is in degree $1$ by homogeneity of the relations. Now, consider $b_2\in A^2$ between $f_1$ and $f_2$, say $b_2: f_1\to f_2$. By the same reasoning, 
\[
	(a, f_2)(i, b_2) = (j, b_2)(a, f_1),
\]
so $(a, f_2)$ is in degree $1$. Continuing in this way, we obtain that $(a, \varepsilon')$ is in degree $1$.
\end{proof}

We conclude this section by showing that being a higher preprojective algebra is a property which is invariant under Morita equivalence. This implies that Theorem \ref{no_tensor} holds also true for algebras that are not basic.  

\begin{proposition}
\label{thm:preproj_morita}
Let $A$ and $B$ be two Morita equivalent algebras. Then $A$ is a higher preprojective algebra if and only if $B$ is a higher preprojective algebra. 
\end{proposition}

\begin{proof}
By Morita theory, there exists progenerators $P:=fA^d$ and $Q:= A^d f$ such that
\[
	B\cong \End_A(P)\cong P\otimes_A Q\quad\text{and}\quad A\cong Q\otimes_B P,
\]
for some integer $d\geq 1$ and full idempotent $f\in A$. Then there is an equivalence of categories:  
\[
	P\otimes_A -: \mathsf{Mod}\, A\xrightarrow{\sim} \mathsf{Mod}\, B.
\]
Suppose that $A$ is a higher preprojective algebra, that is, there exists an algebra $\Lambda\subset A$ which is $n$-representation-infinite such that 
\[
	A \cong T_{\Lambda}\Ext_{\Lambda}^n(D\Lambda, \Lambda).
\]
Because $\Lambda = A_0$ is the degree $0$ part of $A$ with respect to the tensor algebra grading, we must have that $f$ is a full idempotent in $\Lambda$. Thus there is a Morita equivalence 
\[
	\mathcal F := \bar P\otimes_{\Lambda}-\otimes_{\Lambda} \bar Q:\mathsf{Bimod}\, \Lambda\xrightarrow{\sim}\mathsf{Bimod}\, \beta
\]  
where $\bar P:= f\Lambda^d$, $\bar Q:=\Lambda^d f$ and $\beta := \End_{\Lambda}(\bar P)\cong\bar P\otimes_{\Lambda} \bar Q\cong \mathcal F(\Lambda)$ . The functor $\mathcal F$ commutes with the $\Lambda$-bimodule $D\Lambda=\Hom_k(\Lambda,k)$. In fact, 
\begin{align*}
\mathcal F(\Hom_k(\Lambda,k)) &= f\Lambda^d\otimes_{\Lambda} \Hom_k(\Lambda,k)\otimes_{\Lambda} \Lambda^d f\\
&\cong f\Hom_k(\mathcal M_d(\Lambda), k) f\\
&\cong \Hom_k(f\mathcal M_d(\Lambda) f, k)\\
&\cong \Hom_k(\mathcal \beta, k),
\end{align*}
where $\mathcal M_d(\Lambda)$ is the matrix algebra over $\Lambda$. Since $\mathcal F$ is an equivalence of categories which commutes with $k$-duality, we have 
\[
	\mathcal F(\Ext^n_{\Lambda}(D\Lambda, \Lambda)) \cong \Ext^n_{\beta}(D\beta, \beta).
\]
Moreover, for any $N, M\in \mathsf{Bimod}\, \Lambda$  there is an isomorphism of bimodules 
\begin{align*}
	\mathcal F(M\otimes_{\Lambda} N) &= \bar P\otimes_{\Lambda} M\otimes_{\Lambda} N\otimes_{\Lambda} \bar Q\\
	&\cong  \bar P\otimes_{\Lambda} M\otimes_{\Lambda}\Lambda\otimes_{\Lambda} N\otimes_{\Lambda} \bar Q\\
	&\cong  \bar P\otimes_{\Lambda} M\otimes_{\Lambda}\bar Q\otimes_{\beta}\bar P\otimes_{\Lambda} N\otimes_{\Lambda} \bar Q\\
	&\cong \mathcal F(M)\otimes_{\beta}\mathcal F(N).
\end{align*}
Finally, $\mathcal F$ commutes with direct sums, since it is an equivalence. Combining these facts, we obtain
\[
	B\cong \mathcal F(A) \cong T_{\beta}\Ext_{\beta}^n(D\beta, \beta).
\]
Since Morita equivalences preserve the global dimension and finite-dimensionality, we have that $\beta$ is a finite-dimensional algebra of global dimension $n$. Similarly to what we showed above, $\mathcal F$ induces a derived equivalence $\mathcal F: \DD(\Mod \Lambda^e)\xrightarrow{\sim} \DD(\Mod \beta^e)$ which commutes with the inverse Serre functor $\Se^{-1}=\RHom_{\Lambda}(D\Lambda, -)$. Thus 
\[
	\mathcal F(\mathbb S_n^{-\ell}(\Lambda))\cong \mathbb S_n^{-\ell}(\beta),
\]	 
for all $\ell\geq 0$. Since $\Lambda$ is $n$-representation-infinite, $S_n^{-\ell}(\Lambda)$ is concentrated in degree $0$, so $\beta$ is $n$-representation-infinite as well. Therefore, $B$ is a higher preprojective algebra. 
\end{proof}

\begin{corollary}
\label{cor:no_tensor}
If $A = A^1\otimesk A^2$ is the tensor product of two bimodule Calabi--Yau Koszul algebras, then $A$ does not admit the structure of a higher preprojective algebra. 
\end{corollary}

\begin{proof}
For convenience, we first give a proof that the bimodule Calabi--Yau and the Koszulity properties are preserved under Morita equivalence, but this is well-known. Let $B$ and $C$ be two Morita equivalent algebras. Then there exists an equivalence of categories 
\[
	\mathcal F: \Mod B^e\xrightarrow{\sim} \Mod C^e
\]
such that $\mathcal F(B) \cong C$. The functor $\mathcal F$ induces an equivalence 
\[
 	\mathcal F: \DD(\Mod B^e)\xrightarrow{\sim} \DD(\Mod C^e)
\]
which restricts to 
\[
	\tilde {\mathcal F}: \mathsf{per}\,B^e\xrightarrow{\sim} \mathsf{per} \, C^e.
\]
Suppose that $B$ is $n$-bimodule Calabi--Yau. Then $B\in\mathsf{per}\,B^e$, so $C \cong  \tilde{\mathcal F}(B) \in \mathsf{per} \, C^e$ as well. Now the Calabi--Yau duality property (\ref{self_dual})
 of $C$ follows from \cite[Remark 3.4.2]{Gin06}. In fact, applying Van den Bergh duality (\cite[Theorem 1]{VdB98}), we get that the Calabi--Yau property is equivalent to 
\[
	\Tor_{\bullet}^{B^e} (B, -)\cong \Ext_{B^e}^{d-\bullet}(B,-).
\]
Since $\Tor_{\bullet}^{B^e}$ and $\Ext_{B^e}^{\bullet}$ are preserved under equivalences and $\mathcal F(B)\cong C$, we obtain that this property is invariant under Morita equivalence. \\

Now, if $B=T_S(V)/\langle M\rangle$ is Koszul, then $B$ admits a graded bimodule projective resolution $P_{\bullet}$ such that each term $P_{\ell} = B\otimes_S K_{\ell} \otimes_S B$ is generated in degree $\ell$. We use arguments similar to the proof of Proposition \ref{thm:preproj_morita}. The functor $\mathcal F$ induces an equivalence 
\[
	\bar{\mathcal F}: \mathsf{Bimod}\, S\to \mathsf{Bimod}\, \bar S,
\]
where $\bar S = \bar{\mathcal F}(S)$ is a semisimple finite-dimensional algebra. Then 
\[
	C\cong T_{\bar S} \bar{\mathcal F}(V)/\langle \bar{\mathcal F}(M)\rangle
\]
is Koszul, since $\bar{\mathcal F}(P_{\bullet})$ is a bimodule projective resolution of $C$ where $\bar{\mathcal F}(P_{\ell})\cong C\otimes_{\bar S} \bar{\mathcal F}(K_{\ell})\otimes_{\bar S} C$ is generated in degree $\ell$.\\

Let $(-)_b$ denote the basic algebra Morita equivalent to $(-)$. Morita equivalences preserve tensor products over $k$, so we have
\[
	A_b \cong A^1_b\otimesk A^2_b. 
\]
We showed that the basic algebras $A^1_b$ and $A^2_b$ are Koszul bimodule Calabi--Yau. Thus, applying Theorem \ref{no_tensor}, we obtain that $A_b$ does not admit the grading structure of a preprojective algebra. Hence, by Proposition \ref{thm:preproj_morita}, $A$ is not a higher preprojective algebra. 
\end{proof}


\section{Preprojective algebra structure on skew-group algebras}

Our original motivation was to determine whether we can generalize the following theorem, established in \cite{RVdB89} (see also \cite[Corollary 3.5]{CBH98}), in the context of higher Auslander--Reiten theory. 

\begin{theorem}[{\cite[Proof of Proposition 2.13]{RVdB89}}]
\label{thm:RVdB}
Let $G$ be a finite subgroup of $SL(2,k)$. Then the skew-group algebra $k[x,y]\#G$ is Morita equivalent to $\Pi(k\Delta_G)$, the preprojective algebra over the extended Coxeter-Dynkin quiver $\Delta_G$ associated to $G$ via the McKay correspondence.
\end{theorem}

Throughout this section, let $R = k[x_1,\ldots, x_n] = k[V]$ be the polynomial ring in $n$ variables and let $G$ be a finite subgroup of $SL(n, k)$ acting on $R$ by linear change of coordinates. The \emph{skew-group algebra} $R\#G$ is defined as follows:  
\[
	R\#G = R\otimesk kG
\]
as $kG$-modules, and the multplication is given by
\[
	(v_1, g_1)\cdot (v_2, g_2) = (v_1 g_1(v_2), g_1g_2),
\]
where $v_i\in R$ and $g_i\in G$, $i=1,2$. Note that, as described for example in \cite{BSW10}, the skew-group algebra $R\#G$ is isomorphic to
\[
	T_{kG}(V\otimesk kG)/\langle M\otimesk kG\rangle,
\]
where $M$ is the space of commutativity relations in $R = k[V]$, and the bimodule action of $kG$ on $V\otimesk kG$ is given by 
\[
	g_1(v\otimesk h) g_2 = (g_1 v\otimesk  g_1hg_2).
\]

It is natural to ask whether a generalization of Theorem \ref{thm:RVdB} holds, that is, if the skew-group algebra $R\#G$ is Morita equivalent to a higher preprojective algebra for any finite $G<SL(n, k)$. The goal of this section is to give a negative answer to this question.  Namely, we show that if $G$ embeds into a product of special linear groups (Definition \ref{dfn:embed}), then $R\#G$ is not Morita equivalent to a higher preprojective algebra. In the case where $G$ is cyclic, this provides a partial converse to a theorem by Amiot, Iyama and Reiten (cf. Theorem \ref{AIR}). \\

Note that $R\#G$ is always bimodule $n$-Calabi--Yau and Koszul. 

\begin{lemma}[{\cite[Lemma 6.1]{BSW10}}]
\label{lem:resolution}
Let $G<SL(n, k)$ be finite, $R = k[x_1,\ldots, x_n] = k[V]$ and $A:= R\#G$. Then $A$ is bimodule $n$-Calabi--Yau and Koszul. The minimal projective $A$-bimodule resolution of $A$ is given by 
\[
	P_\bullet = A\otimes_{k G}\left(\bigwedge\nolimits^{n} V\otimes_{k}k G\right)\otimes_{k G}A\to A\otimes_{k G}\left(\bigwedge\nolimits^{n-1} V\otimes_{k}k G\right)\otimes_{k G}A
	\to\cdots\to A\otimes_{k G}A
\]
and satisfies $P_{\bullet}\cong P_{\bullet}^{\vee}[n]$, where $(-)^{\vee} = \Hom_{A^e}(-, A^e)$. 
\end{lemma}

Thus the difficulty lies in finding a grading that gives $R\#G$ Gorenstein parameter $1$. Again, this grading must have the property that the degree $0$ part is finite-dimensional, as we want it to be $(n-1)$-representation-infinite. It cannot be induced by putting the variables $x_1, \ldots, x_n$ in $R$ in some degree $a_1, \ldots, a_n$, respectively, and the elements of $G$ in degree $0$, since in this case the Gorenstein parameter would be $\sum_{1\leq i \leq n} a_i$. \\

The algebra structure of $R\#G$ is related to the McKay quiver of $G$. 

\begin{definition}[\cite{Mck80}]
Let $G$ be a finite subgroup of $GL(n, k)$. The \emph{McKay quiver} $Q_G$ of $G$ has a vertex set in which each element corresponds to an irreducible representation $\rho$ of $G$. Let $V$ be the \emph{given representation} of dimension $n$, that is, the representation given by the inclusion  $G\xhookrightarrow{}GL(n, k)$. In addition, let $S_{\rho}$  be the simple $kG$-module corresponding to $\rho$. The number of arrows $\rho_1\to\rho_2$ in $Q_G$ is equal to 
\[
	\dim_k(\Hom_{kG}(S_{\rho_2}, V\otimesk S_{\rho_1})).
\]
\end{definition}

\begin{theorem}[See, e.g., {\cite[Theorem 1.8]{GMV02}}]
\label{GMV}
Let $G$ be a finite subgroup of $GL(n, k)$, then $R\#G$ is Morita equivalent to $kQ_G/\langle M \rangle$, where $Q_G$ is the McKay quiver of $G$ and $\langle M \rangle$ is the $kG$-bimodule of relations induced from the commutativity relations in $R$. 
\end{theorem}

When $G$ is abelian, the Morita equivalence is an isomorphism, since $R\#G$ is basic. There is then an explicit correspondence between the arrows in $Q_G$ and the variables in $R$. In this case, we have a description of the McKay quiver of $G$. The group is simultaneously diagonalizable, so we can assume up to conjugacy that all its elements are diagonal matrices. Let $\rho_i: G\to k^*$ be defined as $\rho_i(g)=\alpha_i$, where $\alpha_i$ is the $i$-th diagonal element of $g$, for $i = 1, \ldots, n$. Then the McKay quiver of $G$ is described as follows.

\begin{lemma}[{\cite[Definition 2.1]{CMT07b}}]
\label{lem:cyclic_mckay}
Let $G$ be an abelian finite subgroup of $GL(n, k)$. The McKay quiver $Q_G$ of $G$ has an arrow $x_i^{\rho}:\rho\to\rho\rho_i$ for each irreducible representation $\rho$ and all $i = 1,\ldots, n$. We say that $x_i^{\rho}$ is an \emph{arrow of type $x_i$}. 
\end{lemma}

The relations are then induced from the relations in $R$. That is, the ideal $\langle M\rangle$ from Theorem \ref{GMV} is given by 
\[
	\langle\{x_i^{\rho\rho_{i'}}x_{i'}^{\rho} - x_{i'}^{\rho\rho_i}x_i^{\rho}\,\, |\,\, 1\leq i\not = i'\leq n\,\,, \rho\mbox{ irreducible} \}\rangle.
\]

If $G$ is a cyclic group of order $r$, then it is generated up to conjugacy by a diagonal matrix $\frac{1}{r} (a_1, \ldots, a_n)$ whose non-zero entries are $\xi^{a_1}, \ldots, \xi^{a_n}$, where $\xi$ is a primitive $r$-th root of unity and $0\leq a_i<r$. We denote this group by 
\[
	G = \langle\frac{1}{r} (a_1, \ldots, a_n)\rangle.
\]


\subsection{Tensor product of skew-group algebras}

We apply the results of last section to the current context. First assume that $G_i$ is a finite subgroup of $SL(n_i, k)$ and $R_i$ is the polynomial ring in $n_i$ variables, for $i = 1, 2$. Then we have the following $k$-module isomorphisms:
\[
	R_1\#G_1\otimesk R_2\#G_2 \cong (R_1\otimesk kG_1)\otimesk (R_2\otimesk kG_2)
	\cong (R_1\otimesk R_2) \otimesk (kG_1\otimesk kG_2)\cong R\otimesk k[G_1\times G_2],
\]
where $R$ is the polynomial ring in $(n: = n_1+ n_2)$ variables. The map is given by 
\[
	(f(v), g_1)\otimesk(g(w), g_2)\mapsto (f(v)\otimesk g(w), (g_1, g_2))
\]
and induces an algebra isomorphism $R_1\#G_1\otimesk R_2\#G_2 \cong R\#(G_1\times G_2)$. Note that $R\#(G_1\times G_2)$ is Morita equivalent to 
\[
	kQ_{G_1}/\langle M_1 \rangle \otimesk kQ_{G_2}/\langle M_2 \rangle, 
\]
where $Q_{G_i}$ is the McKay quiver of $G_i$ and $M_i$ is a set of relations induced from the relations in $R_i$, for $i =1,2$. The following corollary is a particular case of Corollary \ref{cor:no_tensor}.

\begin{corollary}
\label{coro_tensor}
If $G<SL(n, k)$ is a finite group such that $G$ is conjugate to $G_1\times G_2$, where $G_i<SL(n_i, k)$ for some $n_i\geq 1$ such that $n_1+n_2 = n$, then $R\#G$ is not Morita equivalent to a higher preprojective algebra. 
\end{corollary}

This corollary motivates the following definition. 

\begin{definition}
\label{dfn:embed}
Let $G<SL(n, k)$ be finite. We say that $G$ \emph{embeds} into $SL(n_1, k)\times SL(n_2, k)$, for some $n_1, n_2 \geq 1$ such that $n_1+ n_2 = n$, if $G$ is conjugate to a group in which each element is of the form
\[
\left( \begin{array}{cc}
g_1 & 0  \\
0 & g_2  \end{array} \right), 
\]
where $g_1\in SL(n_1, k)$ and $g_2\in SL(n_2, k)$. Equivalently, the given representation $V$ of $G$ can be decomposed as 
\[
	V\cong V_1\oplus V_2, 
\]
where $V_i: G\to SL(n_i, k)$ is a faithful representation, for $i=1,2$. In particular, $\bigwedge^{n_i}V_i$ is the trivial representation. 
\end{definition}

If $\Gamma_1, \Gamma_2$ and $H$ are finite groups and $\phi_i: \Gamma_i\to H$ is an epimorphism for $i = 1,2$, define the \emph{fibre product} $\Gamma_1\times_H \Gamma_2$ as
\[
	\Gamma_1\times_H \Gamma_2: = \left\{\left(\begin{array}{cc}
g_1 & 0  \\
0 & g_2  \end{array} \right)\,\, | \,\, g_1\in \Gamma_1, \, g_2\in \Gamma_2,\, \phi_1(g_1) = \phi_2(g_2)\right\}.
\]
By Goursat's lemma, these are exactly the finite subgroups of $\Gamma_1\times \Gamma_2$. Therefore, $G$ embeds into $SL(n_1, k)\times SL(n_2, k)$ if and only if it is conjugate to a fibre product of the form $G_1\times_H G_2$, where $G_i<SL(n_i, k)$ is finite, for $i=1,2$.

\begin{remark}
\label{emb_rem}
The group $G$ from Corollary \ref{coro_tensor} embeds into $SL(n_1, k)\times SL(n_2, k)$. In the next subsection, we generalize this corollary to any finite group having this property. In this setting, the skew-group algebra is not necessarily a tensor product of skew-group algebras. 
\end{remark}


\subsection{Groups embedding in $SL(n_1, k)\times SL(n_2, k)$}

In their paper, C. Amiot, O. Iyama and I. Reiten proved the following:

\begin{theorem}[{\cite[Theorem 5.6]{AIR15}}]
\label{AIR}
Let $G$ be a finite cyclic subgroup of $SL(n, k)$ of order $r$. Suppose that there exists a generator $g= \frac{1}{r}(a_1, \ldots, a_n)$ of $G$ such that
\begin{equation}
\label{eqn:sumr}
\forall i \,\,\,0< a_i< r \quad\text{ and }\quad a_1+\cdots+ a_n = r. 
\end{equation}
Then there exists a grading on $R\#G$ endowing it with a higher preprojective algebra structure. 
\end{theorem}

\begin{remark}
The authors assume in addition that $(a_i, r ) = 1$ for all $i$. This is equivalent to $R^G$ being an isolated singularity (see \cite[Corollary 2.2]{MS84}). However, it is not needed to establish this specific result. Unless mentioned otherwise, we do not assume that we deal with isolated singularities. Note that in general the sum in condition (\ref{eqn:sumr}) is always a multiple of $r$, since $G< SL(n, k)$. It is easy to see that there always exists a generator for which the sum is equal to $r$ when $n = 2$ and $n=3$ (\cite[Corollary 5.3]{AIR15}).
\end{remark}

Using the description in Lemma \ref{lem:cyclic_mckay} of the McKay quiver $Q_G$ of a cyclic group $G$, we can label its vertices from $0$ to $r-1$. There is an arrow of type $x_j^{\ell}$ from $\ell\to \overline{\ell + a_j}$, the representative in $\{0,\ldots, r-1\}$ of $\ell + a_j\mod r$, for every vertex $\ell\in Q_G$. The grading that defines a preprojective algebra structure on $R\#G$ in Theorem \ref{AIR} is given as follows:
\[
	\mbox{deg}(x_j^{\ell}: \ell\to \overline{\ell+ a_j}) = \left\{
	\begin{array}{ll}
		1  & \mbox{if } \overline{\ell+a_j}<\ell \\
		0 & \mbox{else}.
	\end{array}
\right.
\]

\begin{example}
Let $G = \langle\frac{1}{5}(1,1,3)\rangle$. The McKay quiver of $G$ is given by: 

\begin{center}
\begin{tikzpicture}

    \tikzstyle{ann} = [fill=white,font=\footnotesize,inner sep=1pt]

\def \n {5}
\def \radius {2cm}
\def \second {0.30 cm}
\def \margin {8} 
\def \correction{90}
\def \adjust{56}

\foreach \s in {1,...,\n}
{
  \draw[decoration={markings,mark=at position 1 with {\arrow[scale=2]{>}}},
    postaction={decorate}, shorten >=0.4pt] ({360/\n * (\s - 1) + \correction/\n + \margin}:\radius) 
    to ({360/\n * (\s) + \correction/\n-\margin}:\radius); 
    
}

\foreach \s in {1,...,\n}
{
  \draw[decoration={markings,mark=at position 1 with {\arrow[scale=2]{>}}},
    postaction={decorate}, shorten >=0.4pt] ({360/\n * (\s - 1) + \correction/\n+\margin}:\radius+\second) to ({360/\n * (\s) + \correction/\n-\margin}:\radius+\second); 
}

\foreach \s in {1,...,\n}
{
  \draw[decoration={markings,mark=at position 1 with {\arrow[scale=2]{>}}},
    postaction={decorate}, shorten >=0.4pt] ({360/\n * (\s - 1) + \adjust/\n+\margin-4}:\radius-0.4cm) to ({360/\n * (\s +2)+\adjust/\n+\margin}:\radius-0.2cm);
}

\foreach \s in {0,..., 4}
{
\node[] at ({360/\n * (\s )+ \correction/\n}:\radius + 0.2 cm) {$\s$};
\node[ann] at ({360/\n * (\s+1) + \correction/\n - 180/\n}:\radius-\second) {$x_{1}^{\s}$};
 \node[ann] at ({360/\n * (\s+1) + \correction/\n - 180/\n}:\radius+ \second) {$x_{2}^{\s}$};
}
	\node[ann] at (0.35, -0.4) {$x_3^0$}; 
	\node[ann] at (0, 0.59){$x_3^2$};
	\node[ann] at (0.55, 0.2){$x_3^1$};
	\node[ann] at (-0.55, 0.2){$x_3^3$};
	\node[ann] at (-0.35, -0.4){$x_3^4$};

\end{tikzpicture}    
\end{center}

The algebra $R\# G$ is isomorphic to the path algebra over this quiver with the relations given by 
\[
	\{x_{i'}^{j+a_i} x_{i}^{j}-x_{i}^{j+a_{i'}} x_{i'}^{j}\,\, : \,\, 1\leq i\not = i'\leq 3,\,\, 0\leq j\leq 4\},
\]
where $a_1 = a_2 = 1$ and $a_3 = 3$. Since $G$ satisfies the hypotheses of Theorem \ref{AIR}, we get that $R\#G$ has a structure of preprojective algebra, where the degree $0$ part is given by the following quiver:  

\begin{center}
\begin{tikzpicture}

    \tikzstyle{ann} = [fill=white,font=\footnotesize,inner sep=1pt]

\def \n {5}
\def \radius {2cm}
\def \second {0.30 cm}
\def \margin {8} 
\def \correction{90}
\def \adjust{56}

\foreach \s in {1,...,4}
{
  \draw[decoration={markings,mark=at position 1 with {\arrow[scale=2]{>}}},
    postaction={decorate}, shorten >=0.4pt] ({360/\n * (\s - 1) + \correction/\n + \margin}:\radius) 
    to ({360/\n * (\s) + \correction/\n-\margin}:\radius); 
    
}

\foreach \s in {1,...,4}
{
  \draw[decoration={markings,mark=at position 1 with {\arrow[scale=2]{>}}},
    postaction={decorate}, shorten >=0.4pt] ({360/\n * (\s - 1) + \correction/\n+\margin}:\radius+\second) to ({360/\n * (\s) + \correction/\n-\margin}:\radius+\second); 
}

\foreach \s in {1,...,2}
{
  \draw[decoration={markings,mark=at position 1 with {\arrow[scale=2]{>}}},
    postaction={decorate}, shorten >=0.4pt] ({360/\n * (\s - 1) + \adjust/\n+\margin-2}:\radius-0.2cm) to ({360/\n * (\s +2)+\adjust/\n+\margin}:\radius-0.2cm);
}

\foreach \s in {0,..., 4}
{
\node[] at ({360/\n * (\s )+ \correction/\n}:\radius + 0.2 cm) {$\s$};
}

\foreach \s in {0,...,3}
{
\node[ann] at ({360/\n * (\s+1) + \correction/\n - 180/\n}:\radius-\second) {$x_{1}^{\s}$};
\node[ann] at ({360/\n * (\s+1) + \correction/\n - 180/\n}:\radius+ \second) {$x_{2}^{\s}$};
}
	\node[ann] at (0.3, -0.5) {$x_3^0$}; 
	\node[ann] at (0.47, 0.4){$x_3^1$};
\end{tikzpicture}    
\end{center}

\end{example}

The next lemma shows that condition (\ref{eqn:sumr}) implies that $G$ does not embed into $SL(n_1, k)\times SL(n_2, k)$. A partial converse is also given. 

\begin{lemma}
\label{condition}
Let $G$ be a finite cyclic subgroup of $SL(n, k)$ of order $r$. If the group $G$ embeds into $SL(n_1, k)\times SL(n_2, k)$ for some $n_1, n_2\geq 1$ such that $n_1+n_2 = n$, then exactly one of the following is true.
\begin{enumerate}[(a)]
\item There exists a generator $g = \frac{1}{r}(a_1, \ldots, a_n)$ of $G$ such that $a_i = 0$ for some i; \label{conda}
\item Every generator $g = \frac{1}{r}(a_1, \ldots, a_n)$ of $G$ is such that $0<a_i<r$ and the sum $a_1+\cdots+ a_n > r$. \label{condb}
\end{enumerate}

In particular, $G$ does not satisfy the hypotheses of Theorem \ref{AIR}. We have the following partial converse. If either 

\begin{itemize}
\item $G$ satisfies condition \ref{conda}; or
\item $G$ satisfies condition \ref{condb}, $n = 4$ and $R^G$ is an isolated singularity,
\end{itemize}
then $G$ embeds into $SL(n_1, k)\times SL(n_2, k)$ for some $n_1, n_2\geq 1$ such that $n_1+n_2 = n$.
\end{lemma}

\begin{proof}
Let $G$ be a group embeding into $SL(n_1, k)\times SL(n_2, k)$. Assume that \ref{conda} does not hold. Then, up to conjugacy, the following inequalities hold for every generator $g = \frac{1}{r}(a_1, \ldots, a_n)$ in $G$:  
\[
	a_1+\cdots + a_{n_1}\geq r\quad \text{ and }\quad a_{n_1+1}+\cdots a_n \geq r. 
\]
Thus $a_1+\cdots + a_n \geq 2r >r$. \\

Conversely, if \ref{conda} holds, then $G$ embeds into $SL(n-1, k)\times SL(1,k)$. Finally, if $G<SL(4, k)$ satisfies condition \ref{condb} and $R^G$ is an isolated singularity, then the claim follows from \cite[Theorem 2.4]{MS84}.
\end{proof}

Recall that $A:=R\#G$ is isomorphic to $T_{kG}(V\otimesk kG)/\langle M\rangle$. Before moving further, we describe the superpotential in ${T_{kG}(V\otimesk kG)}$. The minimal $A$-bimodule resolution of $A$ is given in Lemma \ref{lem:resolution}. Note that there is a $kG$-bimodule isomorphism
\begin{align*}
	\bigwedge\nolimits^{\ell} V\otimes_{k}k G\quad\quad&\cong\quad \bigcap_{\mu}\left( (V\otimesk kG)^{\otimes\mu}\otimes_{kG} (M\otimesk kG)\otimes_{kG} (V\otimesk kG)^{\otimes \ell-\mu-2}\right)=:K_{\ell}\\
	x_{j_{\ell}}\wedge x_{j_{\ell-1}}\wedge\cdots\wedge x_{j_1}\otimes 1&\mapsto\qquad\qquad\,\, \sum_{\tau\in\mathfrak{S}_{\ell}} (-1)^{\tau} \tau(x_{j_{\ell}}\otimes x_{j_{\ell-1}}\otimes\cdots\otimes x_{j_1})\otimes 1.
\end{align*}
where the second module is our usual description of the terms in the Koszul resolution, given in Definition \ref{koszul_complex} and the elements of the symmetric group act as always. We use this identification to shorten the notations. The $A$-bimodule generators of 
\[
	P_{\ell} := A\otimes_{k G}\left(\bigwedge\nolimits^{\ell} V\otimes_{k}k G\right)\otimes_{k G}A
\]
are given by the elements $(1,1)\otimes(x_{j_{\ell}}\wedge x_{j_{\ell-1}}\wedge\cdots\wedge x_{1}\otimes 1)\otimes (1,1)$. A superpotential in $R\# G$ is thus given by the $kG$-bimodule generator 
\[
	\omega\otimes 1\subset \bigwedge\nolimits^{n}V\otimesk kG \cong K_{n},
\]
where $\omega := x_n\wedge x_{n-1}\wedge\cdots\wedge x_1$ is the superpotential in $R$ containing the commutativity relations (see \cite[Section 3]{BSW10}). We can view $\omega\otimes 1$ as a linear combination of elements in $T_{kG}(V\otimesk kG)$, using the previous identification. We call these the \emph{summands} of $\omega\otimes 1$. By Lemma \ref{key}, the elements in $\bigwedge^{\ell}V\otimesk kG$ have the same degree as their corresponding elements in $T_{kG}(V\otimesk kG)$.\\

\begin{example}
\label{diagonal}
Let 
\[
	G = \langle\frac{1}{3}(1,2,1,2)\rangle< SL(4, k).
\]
Then 
\[
	G\cong \langle\frac{1}{3}(1,2)\rangle\times_G\langle\frac{1}{3}(1,2)\rangle \hookrightarrow SL(2, k)\times SL(2, k).
\]
We see that every generator satisfies condition \ref{condb} of Lemma \ref{condition}. The McKay quiver is given by

\begin{center}
\begin{tikzpicture}[scale=.6]
 \tikzstyle{every node}=[draw,circle,fill=black,minimum size=5pt,
                            inner sep=0pt]
       \draw [decoration={markings,mark=at position 1 with {\arrow[scale=2]{>}}},
    postaction={decorate}, shorten >=0.4pt] (0.5,0.25) -- (7.5,0.25);  
       \draw [decoration={markings,mark=at position 1 with {\arrow[scale=2]{>}}},
    postaction={decorate}, shorten >=0.4pt] (7.5,-0.25) -- (0.5,-0.25);  
       \draw [decoration={markings,mark=at position 1 with {\arrow[scale=2]{>}}},
    postaction={decorate}, shorten >=0.4pt] (7.5,0.75) -- (0.5,0.75);  
       \draw [decoration={markings,mark=at position 1 with {\arrow[scale=2]{>}}},
    postaction={decorate}, shorten >=0.4pt] (0.5,-0.75) -- (7.5,-0.75);  
    \draw [decoration={markings,mark=at position 1 with {\arrow[scale=2]{>}}},
    postaction={decorate}, shorten >=0.4pt] (8,1.5) -- (4.75,6);  
        \draw [decoration={markings,mark=at position 1 with {\arrow[scale=2]{>}}},
    postaction={decorate}, shorten >=0.4pt] (9.2,1.5) -- (5.95,6); 
        \draw [decoration={markings,mark=at position 1 with {\arrow[scale=2]{>}}},
    postaction={decorate}, shorten >=0.4pt] (4.15,6 ) -- (7.4, 1.5); 
        \draw [decoration={markings,mark=at position 1 with {\arrow[scale=2]{>}}},
    postaction={decorate}, shorten >=0.4pt] (5.35,6) -- (8.6,1.5); 
        \draw [decoration={markings,mark=at position 1 with {\arrow[scale=2]{>}}},
    postaction={decorate}, shorten >=0.4pt] (0.6,1.5) -- (3.85,6);  
        \draw [decoration={markings,mark=at position 1 with {\arrow[scale=2]{>}}},
    postaction={decorate}, shorten >=0.4pt] (-0.6,1.5) -- (2.65,6); 
        \draw [decoration={markings,mark=at position 1 with {\arrow[scale=2]{>}}},
    postaction={decorate}, shorten >=0.4pt] (3.25,6 ) -- (0, 1.5); 
        \draw [decoration={markings,mark=at position 1 with {\arrow[scale=2]{>}}},
    postaction={decorate}, shorten >=0.4pt] (2.05,6) -- (-1.2,1.5); 

 \tikzstyle{every node}=[fill=white,minimum size=5pt,
                            inner sep=0pt]

\draw (3.5, 0.25) node{$x_2^2$};
\draw (4.5, -0.25) node {$x_3^1$};
\draw (2.5, 0.75) node{$x_1^1$};
\draw (5.5, -0.75) node {$x_4^2$};

\draw (7.85, 3.5) node {$x_4^1$};
\draw (6.85, 3.9) node {$x_3^0$};
\draw (5.9, 4.4) node {$x_2^1$};
\draw (4.9, 5) node {$x_1^0$};

\draw (0.25, 3.5) node {$x_4^0$};
\draw (1.25, 4.1) node {$x_3^2$};
\draw (2.25, 4.6) node {$x_2^0$};
\draw (3.25, 5) node {$x_1^2$};

\draw (-0.5,0.2) node{$2$};
\draw (8.5,0.2) node{$1$};
\draw (4 , 7) node{$0$};

\end{tikzpicture}
\end{center}

\noindent The superpotential in $T_{kG}(V\otimesk kG)$ is given by 
\[
	 \sum_{0\leq i\leq 2}\sum_{\tau\in\mathfrak{S}_4} (-1)^{\tau} (x_{\tau^{-1}(4)}^{i+ a_{\tau^{-1}(1)}+a_{\tau^{-1}(2)}+a_{\tau^{-1}(3)}}\otimes x_{\tau^{-1}(3)}^{i+ a_{\tau^{-1}(1)}+a_{\tau^{-1}(2)}}\otimes x_{\tau^{-1}(2)}^{i+a_{\tau^{-1}(1)}}\otimes x_{\tau^{-1}(1)}^i),
\]
and corresponds to $(x_4\wedge x_3\wedge x_2\wedge x_1)\otimes 1$. Here, $a_1 = a_3 = 1$ and $a_2=a_4 = 2$.\\

Comparing with the McKay quiver of $\langle\frac{1}{3}(1,2)\rangle$: 

\begin{center}
\begin{tikzpicture}[scale=.5]
 \tikzstyle{every node}=[draw,circle,fill=black,minimum size=5pt,
                            inner sep=0pt]
       \draw [decoration={markings,mark=at position 1 with {\arrow[scale=2]{>}}},
    postaction={decorate}, shorten >=0.4pt] (1,1) -- (7.05,1);  
       \draw [decoration={markings,mark=at position 1 with {\arrow[scale=2]{>}}},
    postaction={decorate}, shorten >=0.4pt] (7.05,0.5) -- (1,0.5);  
    \draw [decoration={markings,mark=at position 1 with {\arrow[scale=2]{>}}},
    postaction={decorate}, shorten >=0.4pt] (8,1.5) -- (4.75,6);  
     \draw [decoration={markings,mark=at position 1 with {\arrow[scale=2]{>}}},
    postaction={decorate}, shorten >=0.4pt] (4.15,6 ) -- (7.4, 1.5); 
            \draw [decoration={markings,mark=at position 1 with {\arrow[scale=2]{>}}},
    postaction={decorate}, shorten >=0.4pt] (0.6,1.5) -- (3.85,6);  
       \draw [decoration={markings,mark=at position 1 with {\arrow[scale=2]{>}}},
    postaction={decorate}, shorten >=0.4pt] (3.25,6) -- (0,1.5);

 \tikzstyle{every node}=[fill=white,minimum size=5pt,
                            inner sep=0pt]
                            
  \draw (0.3,0.75) node{$2$};
\draw (7.75,0.75) node{$1$};
\draw (4 , 6.75) node{$0$};

\end{tikzpicture}
\end{center}

\noindent we see that there is a doubling of the arrows. As we will show, this property prevents $R\#G$ from having a grading structure of preprojective algebra. The key intuition here is that paths in the superpotential go twice around the same vertex. We will explain how this causes a problem when trying to put a preprojective grading structure on $R\#G$.

\end{example}

The skew-group algebra $R\#G$ is Morita equivalent to the basic algebra $eR\#Ge$, where $e = \sum e_{j}$ is the sum of all idempotents in $kG$ associated to the irreducible representations of $G$. There is an isomorphism  
\[
	eR\#Ge\cong T_{ekGe}(eV\otimesk kGe)/\langle eM\otimesk kG e\rangle,
\]
which can be seen from the proof of Corollory \ref{cor:no_tensor}, but this is well-known. This basic algebra is also given as the path algebra over the McKay quiver $Q_G$ modulo relations. In the case where $G$ is abelian, we have that $e=1$, and so $R\#G$ is already basic.   The superpotential in $eR\#Ge$ is given by 
\[
	e(\omega\otimesk 1)e,
\] 
see \cite[Lemma 2.2]{BSW10} for a proof.\\

From now on, we assume that $G<SL(n,k)$ is a finite group embedding in $SL(n_1, k)\times SL(n_2,k)$ for some $n_1, n_2\geq 1$ such that $n_1+n_2 = n$. We first show that the McKay quivers $Q_G$ are characterized by the property that for each vertex, there exists a summand of the superpotential going twice around it, an example of which is given in \ref{diagonal}. This gives insights as to why it is impossible to put a grading of bimodule Calabi--Yau algebra of Gorenstein parameter $1$ without forcing the degree $0$ part to be infinite-dimensional. In fact, this property is a key ingredient in the proof of Theorem \ref{no_tensor}. \\

Given a path $q:i\to j$, we let $h(q):=j$ and $t(q):=i$. 

\begin{proposition}
\label{prop:double}
Let $G$ be a finite group embedding in $SL(n_1, k)\times SL(n_2,k)$. For every vertex $j$ in $Q_G$, there exists a path $p:j\to j$ which is a summand of the superpotential $e(\omega\otimesk 1) e$ and that can be written as $p = p_2p_1$, where $h(p_1) = t(p_2) = j$.  
\end{proposition}

\begin{proof}
Let $S_j$ be the simple $kG$-module associated to $e_j$.  Each arrow $a:t(a)\to h(a)$ in the McKay quiver can be identified with a basis in 
\[
	e_{h(a)} V\otimesk kG e_{t(a)}\cong \Hom_{kG}(S_{h(a)}, V\otimesk S_{t(a)}).
\]
Theorefore, each path $p: t(p)\to h(p)$ of length $\ell$ corresponds to a morphism 
\[
	\rho: S_{h(p)}\longrightarrow V^{\otimes_S \ell}\otimesk S_{t(p)}.
\]
We also consider the morphism 
\begin{align*}
	\alpha^{\ell}:V^{\otimes_S \ell}\otimesk S_{j}&\longrightarrow \quad\bigwedge\nolimits^{\ell} V\otimesk S_{j}\\
		(v_{\ell}\otimes\cdots \otimes v_1)\otimes s&\longmapsto (v_{\ell}\wedge\cdots\wedge v_1)\otimes s.
\end{align*}
For every vertex $j$, the summand $e_j(\omega\otimesk 1)e_j$ of the superpotential, viewed as an element in $K_n$, induces a morphism 
\begin{align*}
	S_j&\longrightarrow V^{\otimes_S n}\otimesk S_j \\
	e_j&\longmapsto \,\,\,\,\, e_j\omega\otimesk e_j
\end{align*} 
and the composition with $\alpha^n$ is non-zero, since $\omega$ corresponds to $x_n\wedge\cdots \wedge x_1$ via the isomorphism $K_n\cong \bigwedge^n V$. \\
 	
Since $G$ embeds in $SL(n_1, k)\times SL(n_2, k)$, there is a decomposition $V\cong V_1\oplus V_2$ such that $\bigwedge^{n_1} V_1 \cong \bigwedge^{n_2} V_2\cong k$. Now let $p:j\to j$ be a path of length $n$ which is a summand of $e_j(\omega\otimesk 1)e_j$ such that $p = p_2p_1$, where $p_i$ is a path of length $n_i$ in $e_{h(p_i)}V_i^{\otimes_{S} n_i}\otimesk kGe_{t(p_i)}$, for $i=1,2$. Such a decomposition exists because the superpotential is super-cyclically symmetric. Then $p_1$ corresponds to a morphism 
\[
	\rho_1: S_{h(p_1)}\longrightarrow V_1^{\otimes_S n_1}\otimesk S_{j}.
\]
Now the map 
\[
	\alpha^{n_1}\circ\rho_1: S_{h(p_1)}\longrightarrow \bigwedge\nolimits^{n_1} V_1\otimesk S_j\cong S_j
\]
is non-zero since $p$ is a summand of $e_j(\omega\otimesk 1)e_j$. By Schur's lemma, this is an isomorphism, so $j = h(p_1)$. 
\end{proof}

We are now ready to extend the class of subgroups of $SL(n,k)$ for which we can show that the skew-group algebra is not a preprojective algebra.  

\begin{theorem}
\label{important}
Let $G$ be a finite group embedding in $SL(n_1, k)\times SL(n_2,k)$. The skew-group algebra $R\#G$ does not admit a grading structure of $n$-preprojective algebra.
\end{theorem}

\begin{remark}
In the case where $G$ is cyclic, we obtain using Lemma \ref{condition} a partial converse to Theorem \ref{AIR}, which is a full converse if $n\leq 4$.  
\end{remark}

\begin{proof}
Suppose that $G = G_1\times_H G_2$, where $G_1<SL(n_1, k)$ and $G_2< SL(n_2, k)$. Then the given representation decomposes as $V \cong V_1\oplus V_2$, where $V_i$ is an $n_i$-dimensional vector space, $i = 1, 2$. Assume by contradiction that $G$ has a grading structure of preprojective algebra. Then in particular its degree $0$ part is finite-dimensional.\\

By Lemma \ref{gen}, the superpotential $\omega\otimes 1$ is a $kG$-bimodule generator of $\Lambda^n V\otimesk kG$ in degree $1$.  It corresponds to elements of the same degree in $T_{kG}(V\otimesk kG)$, by Lemma \ref{key}. Thus, for any idempotent $e_{\ell}$ corresponding to the irreducible representations, the summand $(x_n\otimes x_{n-1}\otimes\cdots\otimes x_1)\otimes e_{\ell}$ is in degree $1$. Now, using the fact that $G$ embeds in $SL(n_1, k)\times SL(n_2, k)$, we get a decomposition 
\[
	(x_n\otimes x_{n-1}\otimes \cdots \otimes x_1)\otimes e_{\ell} = (x_n\otimes x_{n-1}\otimes \cdots\otimes x_{n_2+1})\otimes e_{\ell}\cdot(x_{n_2}\otimes\cdots\otimes x_1)\otimes e_{\ell},
\]
and each component commutes with the action of $kG$. By additivity of the degrees, one of these two components must be in degree $0$, and the other must be in degree $1$. Suppose without loss of generality that 
\[
	(x_n\otimes x_{n-1}\otimes \cdots\otimes x_{n_2+1})\otimes e_{\ell}
\] 
is in degree $0$. Then, 
\begin{equation*}
	\{1\otimes e_{\ell}, (x_n\otimes x_{n-1}\otimes \cdots\otimes x_{n_2+1})\otimes e_{\ell}, \,(x_n\otimes x_{n-1}\otimes\cdots \otimes x_{n_2+1})^2\otimes e_{\ell}, 
	(x_n\otimes x_{n-1}\otimes \cdots\otimes x_{n_2+1})^3\otimes e_{\ell}, \,\ldots\}
\end{equation*}
is an infinite set of linearly independent elements of degree $0$ in $A$. Therefore, $A$ does not have a structure of preprojective algebra. 
\end{proof} 

\begin{remark}
In (\cite{HIO14}, Question 5.9), the authors conjecture that the quiver of an $(n-1)$-hereditary algebra must be acyclic. If we assume this conjecture to be true, we can use the previous proof to show that if $G$ embeds in a product of special linear groups, then $A:=R\#G$ is not Morita equivalent to a basic preprojective algebra. In fact, by Proposition \ref{prop:double}, there are paths that go twice around each vertex, and the previous proof shows that one of these cycles must be in degree $0$. 
\end{remark}

\begin{example}
Let $G= \langle\frac{1}{3} (1,2,1,2)\rangle<SL(2,k)\times SL(2,k)$ be the group described in Example \ref{diagonal}. The skew-group algebra $R\# G$ does not have a structure of preprojective algebra. By Proposition \ref{prop:double}, paths of length $4$ in the superpotential go twice through the same vertex. However, since they must be in degree $1$, there exists an oriented cycle in degree $0$, which generates an infinite-dimensional subalgebra in degree $0$.
\end{example}

Combining with Proposition \ref{thm:preproj_morita}, we deduce the main theorem of this section. 

\begin{corollary}
\label{cor:main_section5}
Let $G$ be a finite group embedding in $SL(n_1, k)\times SL(n_2,k)$. The skew-group algebra $R\#G$ is not Morita equivalent to a higher preprojective algebra. 
\end{corollary}


\bibliographystyle{halpha}

\bibliography{bib_thibault}

\def\cprime{$'$}
\begin{thebibliography}{HIMO14}

\bibitem[AIR15]{AIR15}
Claire Amiot, Osamu Iyama, and Idun Reiten.
\newblock Stable categories of {C}ohen-{M}acaulay modules and cluster
  categories.
\newblock {\em Amer. J. Math.}, 137(3):813--857, 2015.

\bibitem[Ami09]{Ami09}
Claire Amiot.
\newblock Cluster categories for algebras of global dimension 2 and quivers
  with potential.
\newblock {\em Ann. Inst. Fourier (Grenoble)}, 59(6):2525--2590, 2009.

\bibitem[BGL87]{BGL87}
Dagmar Baer, Werner Geigle, and Helmut Lenzing.
\newblock The preprojective algebra of a tame hereditary {A}rtin algebra.
\newblock {\em Comm. Algebra}, 15(1-2):425--457, 1987.

\bibitem[BGS96]{BGS96}
Alexander Beilinson, Victor Ginzburg, and Wolfgang Soergel.
\newblock Koszul duality patterns in representation theory.
\newblock {\em J. Amer. Math. Soc.}, 9(2):473--527, 1996.

\bibitem[BHon]{BH}
Ragnar-Olaf Buchweitz and Lutz Hille, In preparation.

\bibitem[BS10]{BS10}
Tom Bridgeland and David Stern.
\newblock Helices on del {P}ezzo surfaces and tilting {C}alabi-{Y}au algebras.
\newblock {\em Adv. Math.}, 224(4):1672--1716, 2010.

\bibitem[BSW10]{BSW10}
Raf Bocklandt, Travis Schedler, and Michael Wemyss.
\newblock Superpotentials and higher order derivations.
\newblock {\em J. Pure Appl. Algebra}, 214(9):1501--1522, 2010.

\bibitem[CB99]{CB99}
William Crawley-Boevey.
\newblock Preprojective algebras, differential operators and a {C}onze
  embedding for deformations of {K}leinian singularities.
\newblock {\em Comment. Math. Helv.}, 74(4):548--574, 1999.

\bibitem[CBH98]{CBH98}
William Crawley-Boevey and Martin~P. Holland.
\newblock Noncommutative deformations of {K}leinian singularities.
\newblock {\em Duke Math. J.}, 92(3):605--635, 1998.

\bibitem[CMT07]{CMT07b}
Alastair Craw, Diane Maclagan, and Rekha~R. Thomas.
\newblock Moduli of {M}c{K}ay quiver representations. {II}. {G}r\"obner basis
  techniques.
\newblock {\em J. Algebra}, 316(2):514--535, 2007.

\bibitem[EE07]{EE07}
Pavel Etingof and Ching-Hwa Eu.
\newblock Koszulity and the {H}ilbert series of preprojective algebras.
\newblock {\em Math. Res. Lett.}, 14(4):589--596, 2007.

\bibitem[GI19]{GI19}
Joseph Grant and Osamu Iyama.
\newblock Higher preprojective algebras, koszul algebras, and superpotentials,
  2019, arXiv:1902.07878.

\bibitem[Gin06]{Gin06}
Victor Ginzburg.
\newblock Calabi-{Y}au algebras, 2006, arXiv:math/0612139.

\bibitem[GL87]{GL87}
Werner Geigle and Helmut Lenzing.
\newblock A class of weighted projective curves arising in representation
  theory of finite-dimensional algebras.
\newblock In {\em Singularities, representation of algebras, and vector bundles
  ({L}ambrecht, 1985)}, volume 1273 of {\em Lecture Notes in Math.}, pages
  265--297. Springer, Berlin, 1987.

\bibitem[GMV98]{GMV98}
Edward~L. Green and Roberto Mart{\'{\i}}nez-Villa.
\newblock Koszul and {Y}oneda algebras. {II}.
\newblock In {\em Algebras and modules, {II} ({G}eiranger, 1996)}, volume~24 of
  {\em CMS Conf. Proc.}, pages 227--244. Amer. Math. Soc., Providence, RI,
  1998.

\bibitem[GMV02]{GMV02}
Jin~Yun Guo and Roberto Mart{\'{\i}}nez-Villa.
\newblock Algebra pairs associated to {M}c{K}ay quivers.
\newblock {\em Comm. Algebra}, 30(2):1017--1032, 2002.

\bibitem[GP79]{GP79}
I.~M. Gel{\cprime}fand and V.~A. Ponomarev.
\newblock Model algebras and representations of graphs.
\newblock {\em Funktsional. Anal. i Prilozhen.}, 13(3):1--12, 1979.

\bibitem[Guo11]{Guo11}
Lingyan Guo.
\newblock Cluster tilting objects in generalized higher cluster categories.
\newblock {\em J. Pure Appl. Algebra}, 215(9):2055--2071, 2011.

\bibitem[Hap89]{Hap89}
Dieter Happel.
\newblock Hochschild cohomology of finite-dimensional algebras.
\newblock In {\em S\'eminaire d'{A}lg\`ebre {P}aul {D}ubreil et {M}arie-{P}aul
  {M}alliavin, 39\`eme {A}nn\'ee ({P}aris, 1987/1988)}, volume 1404 of {\em
  Lecture Notes in Math.}, pages 108--126. Springer, Berlin, 1989.

\bibitem[HIMO14]{HIMO14}
Martin Herschend, Osamu Iyama, Hiroyuki Minamoto, and Steffen Oppermann.
\newblock Representation theory of geigle-lenzing complete intersections, 2014,
  arXiv:1409.0668.

\bibitem[HIO14]{HIO14}
Martin Herschend, Osamu Iyama, and Steffen Oppermann.
\newblock {$n$}-representation infinite algebras.
\newblock {\em Adv. Math.}, 252:292--342, 2014.

\bibitem[IO13]{IO13}
Osamu Iyama and Steffen Oppermann.
\newblock Stable categories of higher preprojective algebras.
\newblock {\em Adv. Math.}, 244:23--68, 2013.

\bibitem[Iya11]{Iya11}
Osamu Iyama.
\newblock Cluster tilting for higher {A}uslander algebras.
\newblock {\em Adv. Math.}, 226(1):1--61, 2011.

\bibitem[Kel08]{Kel08}
Bernhard Keller.
\newblock Calabi-{Y}au triangulated categories.
\newblock In {\em Trends in representation theory of algebras and related
  topics}, EMS Ser. Congr. Rep., pages 467--489. Eur. Math. Soc., Z\"{u}rich,
  2008.

\bibitem[Kel11]{Kel11}
Bernhard Keller.
\newblock Deformed {C}alabi-{Y}au completions.
\newblock {\em J. Reine Angew. Math.}, 654:125--180, 2011.
\newblock With an appendix by Michel Van den Bergh.

\bibitem[McK80]{Mck80}
John McKay.
\newblock Graphs, singularities, and finite groups.
\newblock In {\em The {S}anta {C}ruz {C}onference on {F}inite {G}roups ({U}niv.
  {C}alifornia, {S}anta {C}ruz, {C}alif., 1979)}, volume~37 of {\em Proc.
  Sympos. Pure Math.}, pages 183--186. Amer. Math. Soc., Providence, R.I.,
  1980.

\bibitem[MM11]{MM11}
Hiroyuki Minamoto and Izuru Mori.
\newblock The structure of {AS}-{G}orenstein algebras.
\newblock {\em Adv. Math.}, 226(5):4061--4095, 2011.

\bibitem[MS84]{MS84}
David~R. Morrison and Glenn Stevens.
\newblock Terminal quotient singularities in dimensions three and four.
\newblock {\em Proc. Amer. Math. Soc.}, 90(1):15--20, 1984.

\bibitem[Rin98]{Rin98}
Claus~Michael Ringel.
\newblock The preprojective algebra of a quiver.
\newblock In {\em Algebras and modules, {II} ({G}eiranger, 1996)}, volume~24 of
  {\em CMS Conf. Proc.}, pages 467--480. Amer. Math. Soc., Providence, RI,
  1998.

\bibitem[RVdB89]{RVdB89}
Idun Reiten and Michel Van~den Bergh.
\newblock Two-dimensional tame and maximal orders of finite representation
  type.
\newblock {\em Mem. Amer. Math. Soc.}, 80(408):viii+72, 1989.

\bibitem[vdB98]{VdB98}
Michel van~den Bergh.
\newblock A relation between {H}ochschild homology and cohomology for
  {G}orenstein rings.
\newblock {\em Proc. Amer. Math. Soc.}, 126(5):1345--1348, 1998.

\bibitem[Wei94]{Wei94}
Charles~A. Weibel.
\newblock {\em An introduction to homological algebra}, volume~38 of {\em
  Cambridge Studies in Advanced Mathematics}.
\newblock Cambridge University Press, Cambridge, 1994.

\end{thebibliography}


\end{document}